\documentclass{amsart}
\usepackage[active]{srcltx}
\date{October 02, 2015}
\usepackage{graphicx}
\usepackage{url}
\ifx\PROM\undefined\relax
\else
\fi
\usepackage[dvipsnames]{xcolor}

\def\grant{%
  The first author was partly supported by CAPES and JSPS under
  Brazil-Japan research cooperative program, Proc BEX 12998/12-5.
  The second author was
  partly supported by Grant-in-Aid for Scientific Research
  (C) No.\ 26400087
  from the Japan Society for the Promotion of Science,
  the third author by (A) No.\ 26247005 and
  the fourth author by (C) No.\ 26400087 from the Japan
  Society for the
  Promotion of Science.
}
\ifx\PROM\undefined
\title[Behavior of Gaussian and mean curvature]{%
    Behavior of Gaussian curvature and mean curvature
    near non-degenerate
    singular points on wave fronts
}
\author[L. F. Martins]{Luciana F. Martins}
\address[Martins]{%
  Departamento de Matem\'atica,
  IBILCE - UNESP
  R. Cristovao Colombo, 2265, CEP 15054-000,
  Sao Jose do Rio Preto, SP, Brazil}
\email{lmartins@ibilce.unesp.br}
\author[K. Saji]{Kentaro Saji}
\address[Saji]{%
  Department of Mathematics,
  Faculty of Science,
  Kobe University,
  Rokko, Kobe 657-8501, Japan}
\email{saji@math.kobe-u.ac.jp}
\author[M. Umehara]{Masaaki Umehara}
\address[Umehara]{%
  Department of Mathematical and Computing Sciences,
  Tokyo Institute of Technology,
  Tokyo 152-8552, Japan}
\email{umehara@is.titech.ac.jp}
\author[K. Yamada]{Kotaro Yamada}
\address[Yamada]{%
  Department of Mathematics,
  Tokyo Institute of Technology,
  Tokyo 152-8551, Japan}
\email{kotaro@math.titech.ac.jp}
\thanks{\grant}
\else
\title*{
    Behavior of Gaussian curvature and mean curvature
    near non-degenerate
    singular points on wave fronts
}
\titlerunning{Behavior of Gaussian and mean curvature}
\author{L. F. Martins, K. Saji, M. Umehara and K. Yamada}
\authorrunning{L. F. Martins et.\ al.}
\institute{%
   Luciana F. Martins \at 
   Departamento de Matem\'atica,
   IBILCE - UNESP
   R. Cristovao Colombo, 2265, CEP 15054-000,
   Sao Jose do Rio Preto, SP, Brazil,
   \email{lmartins@ibilce.unesp.br}
 \and   
   Kentaro Saji \at
   Department of Mathematics,
   Faculty of Science,
   Kobe University,
   Rokko, Kobe 657-8501, Japan
   \email{saji@math.kobe-u.ac.jp}
 \and
   Masaaki Umehara \at
   Department of Mathematical and Computing Sciences,
   Tokyo Institute of Technology,
   Tokyo 152-8552, Japan,
   \email{umehara@is.titech.ac.jp}
 \and 
   Kotaro Yamada \at
   Department of Mathematics,
   Tokyo Institute of Technology,
   Tokyo 152-8551, Japan,
   \email{kotaro@math.titech.ac.jp}
}
\def\thanks{
  The first author was partly supported by CAPES and JSPS under
  Brazil-Japan research cooperative program, Proc BEX 12998/12-5.
  The second author was
  partly supported by Grant-in-Aid for Scientific Research
  (C) No.\ 26400087
  from the Japan Society for the Promotion of Science,
  the third author by (A) No.\ 26247005 and
  the fourth author by (C) No.\ 26400087 from the Japan
  Society for the
  Promotion of Science.
}
\fi
\ifx\PROM\undefined
\newtheorem{introtheorem}{Theorem}
\newtheorem{introcorollary}[introtheorem]{Corollary}

\newtheorem{thm}{Theorem}[section]
\newtheorem{prop}[thm]{Proposition}
\newtheorem{cor}[thm]{Corollary}
\newtheorem{lem}[thm]{Lemma}
\newtheorem{fact}[thm]{Fact}
\newtheorem*{prob*}{Problem}
\theoremstyle{definition}
\newtheorem{defi}[thm]{Definition}
\newtheorem{rem}[thm]{Remark}
\newtheorem{exa}[thm]{Example}
\newtheorem*{ack}{Acknowledgements}
\newenvironment{tProof}[1]{\begin{proof}[{\rm(}#1{\rm)}]}{\end{proof}}
\newenvironment{Proof}{\begin{proof}}{\end{proof}}
\else
\usepackage{amsmath,amssymb}
\spdefaulttheorem{introtheorem}{Theorem}{\bfseries}{\itshape}
\spnewtheorem{introcorollary}[introtheorem]{Corollary}{\bfseries}{\itshape}

\spnewtheorem{thm}{Theorem}[section]{\bfseries}{\itshape}
\spnewtheorem{prop}[thm]{Proposition}{\bfseries}{\itshape}
\spnewtheorem{cor}[thm]{Corollary}{\bfseries}{\itshape}
\spnewtheorem{lem}[thm]{Lemma}{\bfseries}{\itshape}
\spnewtheorem{fact}[thm]{Fact}{\bfseries}{\itshape}
\spnewtheorem{defi}[thm]{Definition}{\bfseries}{\rmfamily}
\spnewtheorem{rem}[thm]{Remark}{\itshape}{\rmfamily}
\spnewtheorem{exa}[thm]{Example}{\itshape}{\rmfamily}
\spnewtheorem*{ack}{Acknowledgements}{\itshape}{\rmfamily}
\newenvironment{tProof}[1]{%
\begingroup\renewcommand{\proofname}{#1}%
\begin{proof}\smartqed}{\qed\end{proof}\endgroup}
\newenvironment{Proof}{\begin{proof}\smartqed}{\qed\end{proof}}
\fi
\numberwithin{equation}{section}
\newcommand{\op}[1]{{\operatorname{#1}}}
\newcommand{\vect}[1]{{\boldsymbol{#1}}}
\newcommand{\dy}{\displaystyle}
\newcommand{\ccr}{\op{ccr}}
\newcommand{\ext}{\op{ext}}
\newcommand{\sgn}{\op{sgn}}
\newcommand{\R}{\boldsymbol{R}}
\newcommand{\U}{{\mathcal U}}
\renewcommand{\O}{{\mathcal O}}
\newcommand{\M}{{\mathcal M}}
\newcommand{\inner}[2]{\left\langle{#1},{#2}\right\rangle}
\renewcommand{\phi}{\varphi}
\renewcommand{\epsilon}{\varepsilon}
\renewcommand{\det}{\op{det}{}}
\newcommand{\pmt}[1]{{\begin{pmatrix} #1  \end{pmatrix}}}
\begin{document}
\maketitle
\ifx\PROM\undefined
\begin{abstract}
 We define cuspidal curvature $\kappa_c$ 
 (resp.\ normalized cuspidal curvature $\mu_c$)
 along cuspidal edges 
 (resp.\  at a swallowtail singularity)
 in Riemannian $3$-manifolds,
 and show that it gives a coefficient
 of the divergent term
 of the mean curvature function. 
 Moreover, we show that the product 
 $\kappa_\Pi^{}$ called the product curvature 
 (resp.\ $\mu_\Pi^{}$ called normalized product curvature)
 of  $\kappa_c$ (resp.\ $\mu_c$)
 and the limiting normal curvature $\kappa_\nu$
 is an intrinsic invariant of the surface, and is
 closely related to
 the boundedness of the Gaussian curvature.
 We also consider the limiting behavior of
 $\kappa_\Pi^{}$ when cuspidal edges
 accumulate to other singularities.
 Moreover, several new
 geometric invariants of cuspidal edges 
 and swallowtails  are given.
\end{abstract}
\else
\abstract{%
 We define cuspidal curvature $\kappa_c$ 
 (resp.\ normalized cuspidal curvature $\mu_c$)
 along cuspidal edges 
 (resp.\  at a swallowtail singularity)
 in Riemannian $3$-manifolds,
 and show that it gives a coefficient
 of the divergent term
 of the mean curvature function. 
 Moreover, we show that the product 
 $\kappa_\Pi^{}$ called the product curvature 
 (resp.\ $\mu_\Pi^{}$ called normalized product curvature)
 of  $\kappa_c$ (resp.\ $\mu_c$)
 and the limiting normal curvature $\kappa_\nu$
 is an intrinsic invariant of the surface, and is
 closely related to
 the boundedness of the Gaussian curvature.
 We also consider the limiting behavior of
 $\kappa_\Pi^{}$ when cuspidal edges
 accumulate to other singularities.
 Moreover, several new
 geometric invariants of cuspidal edges 
 and swallowtails  are given.
}
\begingroup
\renewcommand{\thefootnote}{\relax}%
\footnotetext{\thanks}
\endgroup
\fi

\section*{Introduction}
In \cite{SUY}, the behavior of the 
Gaussian curvature $K$ along cuspidal
edge singularities in Riemannian $3$-manifolds 
$(M^3,g)$ was discussed.
However, the existence of \lq intrinsic\rq\ 
invariants related to the
boundedness of $K$ was not mentioned there. 
In this paper, we show that
several given invariants 
of 
singularities of surfaces in $M^3$
are actually intrinsic%
\footnote{
  These invariants 
  can be treated as invariants
  of a certain class of
  positive semi-definite metrics,
  see \cite{HHNSUY}, \cite{NUY} and 
  \cite{SUY4}.},  
by
\begingroup
\renewcommand{\theenumi}{{\rm{(\roman{enumi})}}}
\renewcommand{\labelenumi}{{\rm{(\roman{enumi})}}}
\begin{enumerate}
 \item setting up a class of local coordinate systems determined by the
       induced
       metrics (i.e.\ the first fundamental forms),
 \item and giving formulas for the invariants in terms 
       of the coefficients of the first
       fundamental forms with respect to the above coordinate systems.
\end{enumerate}
\endgroup%
Recently, geometric invariants of
cross cap singularities on surfaces are discussed by
several geometers
(\cite{BW, dt,fh,fh2,ggs,HHNUY,Oset-Tari,faridtari,west}). 
Also,
Nu\~no-Ballesteros and the first 
author \cite{MB} investigated
geometric properties of rank one singularities other than cross
caps. After that, in a joint work \cite{MS} of the first two
authors, a normal form for cuspidal edges was given and
geometric meanings of its coefficients were discussed
as invariants of cuspidal edges. 
For example, 
the singular curvature $\kappa_s$
and
the limiting normal curvature $\kappa_\nu$
for cuspidal edge singularities 
are defined in  \cite{SUY},
each of which appears as one of 
these coefficients of the normal form (cf.\ \cite{MS}).

In this paper, 
we generalize the concept of limiting normal curvature
$\kappa_\nu$
for an arbitrarily given rank one
singular point on each wave front.
(The definition of wave fronts (or fronts for short)
is given in Section~\ref{sec:prelim}.)
A non-degenerate singular point $p$ 
(cf.\ Definition \ref{def:lambda})
on a front is a 
rank one singular point such that
the component of singular set
containing $p$ 
consists of a regular curve
in the source space, which is called
the \lq singular curve\rq\ 
(see Section~\ref{sec:prelim}).
In the Euclidean 3-space $\R^3$, the Gauss maps
of fronts are defined and
smoothly extended across the singular curve.
One of our main results is as follows:
\begin{introtheorem}\label{thm:A}
 Let $\U$ be
 a domain in $\R^2$, and
 $f:\U\to (M^3,g)$  a front
 which admits only
 non-degenerate singular points.  
 Then the $2$-form $K\,d\hat A$ 
 can be smoothly extended to
 $\U$,
 where $d\hat A=\det_g(f_u,f_v,\nu)\,d u\wedge dv$ 
 is the signed area element
 {\rm(}cf.\ Remark \ref{rem:KdA}{\rm).
 Moreover, for each singular point $p$ on $\U$},
 the following two conditions are equivalent{\rm :}
 \begingroup
 \renewcommand{\theenumi}{{\rm(\arabic{enumi})}}
 \renewcommand{\labelenumi}{{\rm(\arabic{enumi})}}
 \begin{enumerate} 
  \item\label{item:A1} 
       the limiting normal curvature $\kappa_\nu$
       at $p$ is equal to zero,
  \item\label{item:A3} 
       the extension of the $2$-form $K\,d\hat A$ 
       vanishes at $p$.
 \end{enumerate}
 \endgroup
If $\kappa_\nu(p) \ne 0$, the Gaussian curvature $K$
is unbounded near $p$ and changes sign between
two sides of the singular curve.
 Furthermore, if $(M^3,g)$ is the Euclidean $3$-space,
 the above two conditions are equivalent to
 that the Gauss map $\nu:\U\to S^2$ of $f$ has a 
 singularity at $p$.
\end{introtheorem}
In Section~\ref{sec:cuspidal},
we newly introduce the {\it cuspidal curvature\/}
$\kappa_c$ along cuspidal edges.
We show that $\kappa_c$
coincides with the cuspidal curvature of
the cusp of the plane curve
obtained as the section of the
surface by the plane $P$,
where $P$ is the plane
orthogonal to the 
tangential direction 
at a given cuspidal edge
(cf.\ Proposition \ref{prop:geom0}).
The cuspidal curvature $\kappa_c$
appears in the first coefficient of
the divergent term  of the mean curvature function (cf.\ \eqref{eq:hatH}).
Then, we consider the product
\[
    \kappa_\Pi^{}:=\kappa_\nu \kappa_c
\]
called {\it product curvature\/}
along cuspidal edges, and show that 
it is an \lq intrinsic invariant\rq. 
Using  $\kappa_\Pi^{}$, we give a necessary and sufficient
condition for the boundedness of the Gaussian curvature function around
cuspidal edges (cf.\ Corollary~\ref{cor:cusp} and Theorem~\ref{thm:main1}). 
Similarly, in Section~\ref{sec:general}, we also define
the {\it normalized cuspidal curvature\/} $\mu_c$
as the first coefficient of
the divergent term  of the mean curvature function
at swallowtail singularities, and
consider the product
\[
    \mu_\Pi^{}:=\kappa_\nu \mu_c
\]
called {\it normalized product curvature},
which is an \lq intrinsic invariant\rq\ 
of swallowtail singularities and 
related to the boundedness of the Gaussian curvature 
function (cf.\ Proposition \ref{prop:gaussbdd}). 
We then discuss the limiting behavior of
$\kappa_c$ and $\kappa_\nu$
when cuspidal edges accumulate to other singularities,
in Section \ref{sec:general}.  As a
consequence, we get the following property of the limiting normal
curvature:
\begin{introtheorem}\label{thm:B}
 Let $f:\U\to (M^3,g)$ be a front, and $p\in \U$
 a non-degenerate singular point, where $\U$ is
 a domain in $\R^2$ and $(M^3,g)$ is 
 a Riemannian $3$-manifold.
 Then the Gaussian curvature of $f$
 is rationally bounded at $p$ if and only if
 the limiting normal curvature $\kappa_\nu(p)$
 is equal to zero.
\end{introtheorem}
The definition of rational boundedness is given in Definition
\ref{def:rational}.

This assertion is a consequence of
Corollary \ref{cor:cusp} and Theorem \ref{thm:gaussbdd}.
In Theorem \ref{thm:B}, the assumption that
$f$ is a front is crucial (see Example \ref{ex:crcp}).
The above two theorems yield the
following assertion, which summarizes 
the geometric properties of the
limiting normal curvature:
\begin{introcorollary}\label{cor:C}
 Let $f:\U\to \R^3$ be a front, and $p\in \U$
 a non-degenerate singular point, where $\U$ is
 a domain in $\R^2$. Then the following
 three properties are equivalent{\rm :}
 \begingroup
 \renewcommand{\theenumi}{{\rm(\arabic{enumi})}}
 \renewcommand{\labelenumi}{{\rm(\arabic{enumi})}}
 \begin{enumerate}
  \item\label{item:C1} 
        the Gaussian curvature of $f$
	is rationally bounded at $p$,
  \item\label{item:C2} 
        the limiting normal curvature at $p$
	is equal to zero,
  \item\label{item:C3} 
        a singular point $p$ of $f$
	is also a singular point of
	the Gauss map of $f$.
 \end{enumerate}
 \endgroup
\end{introcorollary}
In \cite[Lemma 3.25]{SUY3}, the 
second, third, forth authors showed that
the singular set of the Gauss map coincides with
the singular set of $f$ if $\log |K|$ is bounded.
The equivalency of \ref{item:C1} and \ref{item:C3} is 
a refinement of it.

At the end of this paper, we introduce 
a new invariant of swallowtails called
{\it limiting singular curvature\/} $\tau_s$,  
which is related to the cuspidal curvature 
of the orthogonal 
projection of the swallowtail singularities 
(cf.\ Corollary~\ref{cor:projection}).

\section{Limiting normal curvature}\label{sec:prelim}
Let $\Sigma^2$ be an oriented $2$-manifold and 
$f:\Sigma^2\to (M^3,g)$ a $C^\infty$-map into an oriented Riemannian 
$3$-manifold $(M^3,g)$. 
A {\it singular point\/} of $f$ is a point at which $f$ is not an
immersion. 
The map $f$ is called a {\it frontal\/} if for each $p\in \Sigma^2$,
there exist a neighborhood $\U$ of $p$ and a unit vector field $\nu$
along $f$ defined on $\U$ such that $\nu$ is perpendicular to
$df(\vect a)$ for all tangent vectors $\vect a\in T\U$. Moreover, if 
$\nu:\U\to
TM^3$ gives an immersion, $f$ is called a (wave) {\it front}. 
We fix a frontal $f:\Sigma^2\to M^3$.
\begin{defi}\label{def:lambda}
 A singular point $p\in \Sigma^2$ of a frontal $f$
 is called {\it non-degenerate\/} if
 the exterior derivative of the function
 \[
   \lambda:=\det_g(f_u,f_v,\nu)\quad
    \bigl(
     f_u := df(\partial_u),~
     f_v := df(\partial_v)
    \bigr)
 \]
 does not vanish at $p$, where $(u,v)$
 is a local coordinate system of $\Sigma^2$ at $p$,
 $\partial_u=\partial/\partial u$,
 $\partial_v=\partial/\partial v$,
 and $\det_g$ is the Riemannian volume form
 of $(M^3,g)$.
 Here, the function $\lambda$ is called
 the {\it signed area density function\/} with respect to
 the local coordinate system $(u,v)$ of $\Sigma^2$.
 If the ambient space $(M^3,g)$ is the
 Euclidean $3$-space $\R^3$, then 
 \lq $\det_g$\rq\ can be
 identified with the usual determinant
 function \lq$\det$\rq\ for
 $3\times 3$-matrices.
\end{defi}
\begin{rem}\label{rem:KdA}
 We set
 \begin{equation}\label{eq:KdA}
  (d\hat A=)d \hat A_f:=\lambda\, du\wedge dv=\det_g(f_u,f_v,\nu)\,du\wedge dv,
 \end{equation}
 which is called the {\it signed area element\/} 
 of $f$ defined in
 \cite{SUY} and \cite{SUY3}.
 If the ambient space is the
 Euclidean $3$-space $\R^3$,
 the Gauss map $\nu$ takes values in the unit sphere
 $S^2$. By the Weingarten formula,
 we have
 \[
    (\nu_u,\nu_v)=-(f_u,f_v) W,
    \qquad
       W:=
          \pmt{g_{11} & g_{12} \\ 
               g_{12} & g_{22}}^{-1}
          \pmt{h_{11} & h_{12} \\
               h_{12} & h_{22}},
 \]
 where 
 \[
    ds^2=g_{11}\,du^2+2g_{12}\,du\,dv+g_{22}\,dv^2,\quad
    h=h_{11}\,du^2+2h_{12}\,du\,dv+h_{22}\,dv^2
 \]
 are the first and the second fundamental forms
 of $f$.
 Since $\nu$ itself can be considered as 
 the unit normal vector of the Gauss map $\nu$,
 we have
 \begin{align*}
  d \hat A_\nu
  &=\det(\nu_u,\nu_v,\nu)
  =\det(-a^1_1f_u-a^2_1f_v,-a^1_2f_u-a^2_2f_v,\nu)
  \\
  &=\det(W)\det(f_u,f_v,\nu)=K\, d\hat A_f,
 \end{align*}
 where $W=(a^i_{j})_{i,j=1,2}$ and
 $K$ is the Gaussian curvature of $f$.
 (We used the Gauss equation $K=\det(W)$.)
 This fact implies that $K\,d \hat A_f$ coincides with 
 the pull-back of the area element of $S^2$.
 In particular, $K\,d \hat A_f$ can be smoothly
 extended across the singular set.
\end{rem}

A non-degenerate singular point $p$ of $f$ is a rank one
singular point, i.e., the kernel of $df(p)$ is  one dimensional.
Since $\{\lambda=0\}$ is the singular set, by the implicit function
theorem, we can take a regular curve $\gamma(t)$ ($|t|<\varepsilon$) 
on $\Sigma^2$ as a parametrization of the singular set
such that $\gamma(0)=p$, where $\varepsilon>0$.
(We call $\gamma$ the {\it singular curve}.) 
There exists a non-vanishing vector field
$\eta(t)$ along $\gamma$ such that $df(\eta(t))$ vanishes
identically. We call $\eta(t)$ a {\it null vector field\/} along
$\gamma$.

A non-degenerate singular point $p$ is
said to be of the {\it first kind\/} if $\eta(0)$ is not
proportional to $\gamma'(0):=d\gamma/dt|_{t=0}$.
Otherwise, it is said to be
of the {\it second kind\/}.
\begin{defi}
 A singular point $p\in \Sigma^2$ of a map $f:\Sigma^2\to M^3$
 is a {\em cuspidal edge}\/
 if the map germ $f$ at $p$ is right-left equivalent to
 $(u,v)\mapsto(u,v^2,v^3)$ at the origin.
 A singular point $p$ of $f$
 is a {\em swallowtail\/} (respectively, {\em cuspidal cross cap})
 if $f$ at $p$ is right-left equivalent to
 $(u,v)\mapsto(u,4v^3+2uv,3v^4+uv^2)$
 (respectively, $(u,v)\mapsto(u,v^2,uv^3)$) at the origin.
 Here, $f$ is considered as a map germ $f\colon{}(\R^2,0)\to(\R^3,0)$
 by taking local coordinate systems of $\Sigma^2$ and $M^3$
 at $p$ and $f(p)$, respectively,
 and
 two map germs $f_1$ and $f_2$ are
 {\em right-left equivalent\/} if there exist diffeomorphism
 germs
 $\phi:(\R^2,0)\to (\R^2,0)$ and
 $\Phi:(\R^3,0)\to (\R^3,0)$
 such that
 $\Phi\circ f_1=f_2\circ \phi$ holds.
\end{defi}
Figures of these singularities are shown in Fig.\ \ref{fig:ceswccr}.
There are criteria for these singularities.

\begin{figure}[htbp]
\centering
 \begin{tabular}{c@{\hspace{3em}}c@{\hspace{3em}}c}
  \includegraphics[width=.25\linewidth]{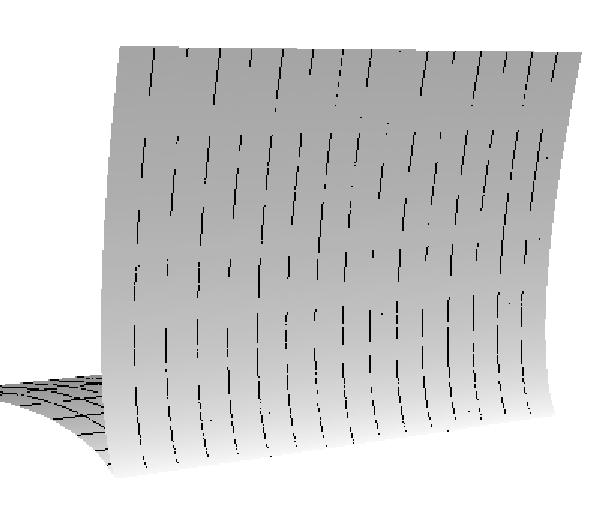}&
  \includegraphics[width=.25\linewidth]{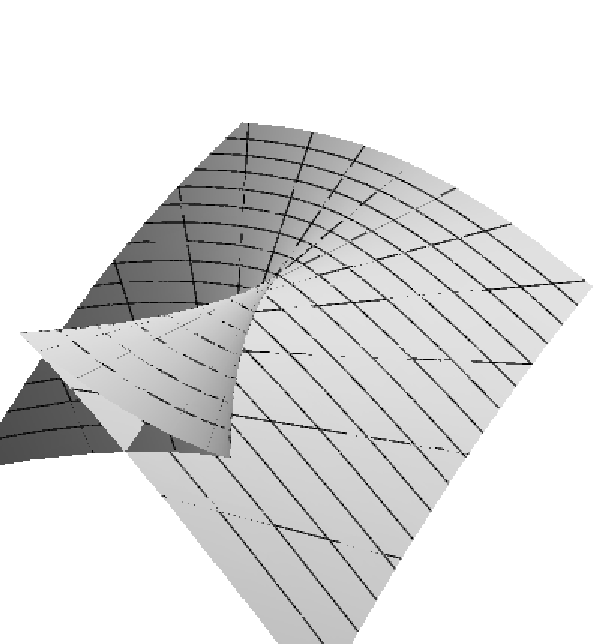}&
  \includegraphics[width=.25\linewidth]{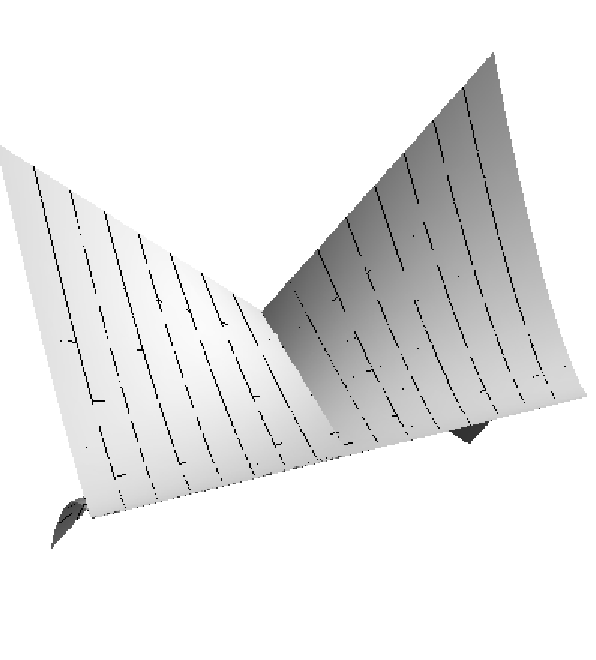}
 \end{tabular}
 \caption{A cuspidal edge,
          a swallowtail and
          a cuspidal cross cap.}
 \label{fig:ceswccr}
\end{figure}
\begin{fact}%
 [{\cite[Proposition 1.3]{KRSUY}, \cite[Corollary 1.5]{FSUY}
   see also \cite[Corollary 2.5]{SUYcamb}}]
\label{fact:criteria}
 Let
 $f\colon{}\Sigma^2\to (M^3,g)$ be a frontal and $p$
 a non-degenerate singular point.
 Take the singular curve $\gamma(t)$ such that $\gamma(0)=p$
 and a null vector field $\eta(t)$ along $\gamma$.
 Then
 \begingroup
 \renewcommand{\theenumi}{{\rm (\arabic{enumi})}}
 \renewcommand{\labelenumi}{{\rm(\arabic{enumi})}}
 \begin{enumerate}
 \item\label{item:Fact1}  
        $f$ at $p$ is a cuspidal edge
	if and only if $f$ is a front 
        and $\{\gamma',\eta\}$ is linearly independent
	at $p$, that is, $p$ is of the first kind,
  \item\label{item:Fact2}  
        $f$ at $p$ is a swallowtail
	if and only if $f$ is a front and
        $\{\gamma',\eta\}$ is linearly dependent 
        at $p$ {\rm(}i.e.,\ $p$ is of the second kind{\rm)},
	but $(d/dt)\det(\gamma'(t),\eta(t))|_{t=0}\ne0$
	holds,
	where $\det$ denotes an area element of $\Sigma^2$, and
  \item\label{item:Fact3}  
        $f$ at $p$ is a cuspidal cross cap
	if and only if
       $\{\gamma',\eta\}$ is     
        linearly independent at $p$
        {\rm(}i.e., $p$ is of the first kind{\rm)},
	$\psi_{\ccr}(0)=0$ and $\psi'_{\ccr}(0)\ne0$,
	where
	\[
	\psi_{\ccr}(t):=\det_g\biggl(
	\hat\gamma'(t),\ \nu\bigl(\gamma(t)\bigr),\
	(\nabla_{\eta}\nu)\bigl(\gamma(t)\bigr)
	\biggr)
	\quad
	\biggl(\hat\gamma(t):=f\bigl(\gamma(t)\bigr)\biggr),
	\]
	$\nu$ is the unit normal vector field,
	and
	$\nabla$ is the Levi-Civita connection of $(M^3,g)$.
 \end{enumerate}
 \endgroup
\end{fact}
Regarding $\psi_{\ccr}$,
the following lemma holds.
\begin{lem}[{\cite[Corollary 1.7]{FSUY}}]
\label{lem:psiccr}
 Let $f: \Sigma^2\to (M^3,g)$ be
 a frontal  and $p$ a singular point of the first kind.
 Then $f$ is a front at $p$ if and only if
 $\psi_{\ccr}\ne0$ at $p$.
\end{lem}
\begin{Proof}
 Since $\eta$ is a null vector and $p$ is
 a singular point of rank one of
 $f$,   $f$ is a front at $p$ if and only if $\nabla_{\eta}\nu(p)\ne0$.
 Thus the necessity is obvious.
 Let us assume that $\psi_{\ccr}=0$ at $p$.
 Since $p$ is a singular point of the first kind,
 $\hat\gamma'\ne0$ holds at $p$.
 In particular,
 $\nabla_{\eta}\nu$ is a linear combination of
 $\hat \gamma'$ and $\nu$ at $p$.
 We now take a local coordinate system
 $(u,v)$ centered at $p$, such that the $u$-axis
 is a singular curve, and $v$-directions
 are null directions along the $u$-axis.
 Noticing that $\nabla$ is the Levi-Civita connection,
 it holds that
 \begin{align*}
  \inner{\nabla_{\eta}\nu}{f(\gamma(t))'}
  &=\inner{\nabla_{\partial_v}\nu}{df(\partial_u)} 
=
  (\inner{\nu}{f_u})_v
  -\inner{\nu}{\nabla_{\partial_v}df(\partial_u)}\\
  &=
  -\inner{\nu}{\nabla_{\partial_u}df(\partial_v)}
=
  -\inner{\nu}{df(\partial_v)}_u
  +\inner{\nabla_{\partial_u}\nu}{df(\partial_v)}
  =0,
 \end{align*}
 where $\inner{~}{~}$ is the inner product corresponding
 to the Riemannian metric $g$,
 and we used the fact that $df(\partial_v)=0$ since
 $\eta=\partial_v$ is the null direction.
 On the other hand, it holds that
 $\inner{\nabla_{\eta}\nu}{\nu}=0$ at $p$.
 Hence $\nabla_{\eta}\nu(p)=0$ holds.
\end{Proof}

By Fact~\ref{fact:criteria}, if $f$ is a front, then a
singular point of the first kind
gives a cuspidal edge.
When $f$ is a frontal but not a front, cuspidal cross caps are
typical examples of non-degenerate singular
points of the first kind.
Swallowtails are singularities of the second kind.
Moreover, if $f$ is a front, \lq non-degenerate peaks\rq\ in the sense of
\cite[Definition 1.10]{SUY} are also singular points of the second kind.

\begin{exa}
 Let
 $f_1\colon{}\R^2\to\R^3$ be the map
 defined by
 \[
    f_1(u,v):=(5u^4+2uv,v,4u^5+u^2v-v^2).
 \]
 Then the singular set is
 $\{10u^3+v=0\}$, the null vector field is
 $\eta=\partial_u$ 
 and $d\nu(\partial_u)=(-1,0,0)$ holds at the origin,
 where $\nu$ is the unit normal vector field.
 Hence
 $f_1$ is a front and $(0,0)$ is a singular point of the second kind
 (namely, a non-degenerate peak in the sense of \cite[Definition 1.10]{SUY})
 but not a swallowtail
 (see the left hand side of Fig.\ \ref{fig:2nd}).
 On the other hand,
 let $f_2\colon{} \R^2\to\R^3$ be the map
 defined by
 \[
   f_2(u,v):=(u^2+2v,u^3+3uv,u^5+5u^3v).
 \]
 Then the singular set is
 $\{v=0\}$, the null vector field is
 $\eta=\partial_u-u\partial_v$,
 and $d\nu(0,0)=(0,0,0)$ holds.
 Hence  $f_2$ is a frontal but not a front at the origin
 $(0,0)$.
 In fact, $(0,0)$  is a singular point of 
 the second kind (see the right hand side of Fig.\ \ref{fig:2nd}).
\end{exa}
\begin{figure}[htbp]
\centering
\begin{tabular}{c@{\hspace{7em}}c}
 \includegraphics[width=.25\linewidth]{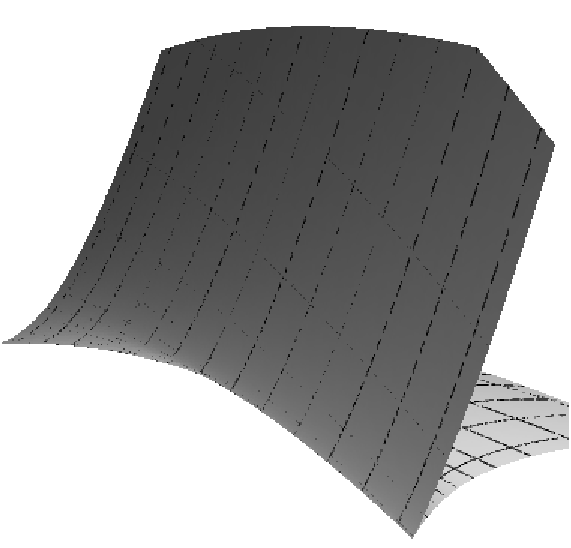}&
 \includegraphics[width=.25\linewidth]{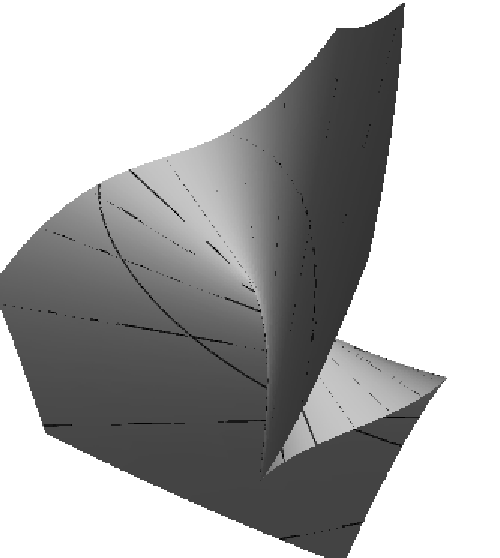}
\end{tabular}
\caption{Singularities of the second kind.}
\label{fig:2nd}
\end{figure}

\begin{defi}\label{def:admissible}
 Let $p$ be a rank one singular point of a frontal 
 $f\colon{}\Sigma^2\to M^3$.
 A local coordinate system $(\U;u,v)$ centered at $p$ is
 called {\it admissible\/} if $f_v(p)=0$
 and it is compatible with respect to
 the orientation of $\Sigma^2$.
\end{defi}

We denote by \lq$\inner{~}{~}$\rq\ the inner product on $M^3$
induced by the metric $g$, and 
$|\vect{a}|:=\sqrt{\inner{\vect{a}}{\vect{a}}}$
($\vect{a}\in TM^3$).
We fix a unit normal vector field $\nu$ of $f$.
Take an admissible coordinate system at a
rank one singular point $p$.
Then we define
\begin{equation}\label{eq:kn1}
 \kappa_\nu(p):=\frac{\inner{f_{uu}(p)}{\nu(p)}}{|f_u(p)|^2},
\end{equation}
which is called the {\it limiting normal curvature}.
Here, we denote $f_{uu}:=\nabla_{\partial_u}f_u$.
The definition of $ \kappa_\nu(p)$
does not depend on the choice of the admissible
coordinate system (see Proposition~\ref{prop:conti}).  
Moreover, $\kappa_\nu(p)$ does not depend on the choice of
the orientation of the source domain,
but depends on the co-orientability
(i.e., the $\pm$-ambiguity of $\nu$).
This definition \eqref{eq:kn1} of $ \kappa_\nu$ is
a generalization of the previously defined
limiting normal curvature 
\begin{equation}\label{eq:knp}
  \kappa_\nu(\gamma(t)):=
   \frac{\inner{\hat\gamma''(t)}
   {\nu(\gamma(t))}}{|\hat \gamma'(t)|^2}
\end{equation}
given in \cite{SUY}, where $\gamma(t)$
is a singular curve parameterizing 
cuspidal edges and 
\begin{equation}\label{eq:gamma-diff}
 \hat\gamma:=f\circ\gamma,\qquad
 \hat\gamma':=\frac{d}{dt} \hat\gamma,\qquad
 \hat\gamma'':=\nabla_{\hat\gamma'}\hat\gamma'.
\end{equation}
In fact, the following assertion holds:
\begin{prop}\label{prop:old}
 The new definition \eqref{eq:kn1} of
 the limiting normal curvature at $p$
 coincides with $\kappa_\nu$
 given in \eqref{eq:knp}
 when $p=\gamma(0)$ is a 
 non-degenerate
 singular point of the first kind.
\end{prop}
\begin{Proof}
 Since $\gamma(t)$ consists of non-degenerate
 singular points of the first kind
 for sufficiently small $t$,
 we can take an admissible coordinate system $(u,v)$
 centered at $\gamma(0)$
 so that the $u$-axis is a singular curve and
 $\partial_v$ is the null direction.
 Since $f(u,0)=\hat \gamma(u)$,
 \eqref{eq:kn1} is exactly equal to \eqref{eq:knp}
 at $t=0$.
\end{Proof}

\begin{prop}[The continuity of the limiting normal curvature]
\label{prop:conti}
 Let $p$ be a rank one singular point of $f$.
 The definition of $\kappa_\nu$
 does not depend on the choice of the admissible coordinate
 system. 
 Moreover, if $p$ is non-degenerate and
 $\gamma(t)$ is
 a singular curve such that $\gamma(0)=p$,
 and if $\gamma(t)$ $(t\ne 0)$ consists of
 singular points of the first kind,
 then it holds that
 \begin{equation}\label{eq:kn2}
  \kappa_\nu(p)
  \left(=\lim_{t\to0}\kappa_{\nu}(\gamma(t))\right)
  =
  \lim_{t\to 0}
   \frac{\inner{\hat \gamma''(t)}
   {\nu(\gamma(t))}}{|\hat \gamma'(t)|^2},
 \end{equation}
 where $\hat \gamma(t):=f(\gamma(t))$ and
 $\hat\gamma''(t):=\nabla_{\hat\gamma'(t)} \hat\gamma'(t)$.
\end{prop}
\begin{Proof}
 Let $(U,V)$ be another admissible coordinate system.
 Then $U_u(p)\ne 0$ holds.
 Moreover,
 \begin{align*}
  \inner{f_{uu}}{\nu}
  &=-\inner{f_{u}}{\nu_u} =
  - U_u \inner{f_U}{U_u\nu_U +V_u\nu_V}\\
  &=-U_u^2\inner{f_U}{\nu_U}-U_uV_u\inner{f_U}{\nu_V}
 \end{align*}
 holds at $p$,
 where $\nu_u=\nabla_{\partial_u}\nu$, etc.
 Since
\begin{equation}\label{eq:fuv}
  \inner{f_U}{\nu_V}=-\inner{f_{UV}}{\nu}=
   -\inner{f_{VU}}{\nu}=\inner{f_V}{\nu_U}=0
\end{equation}
 at $p$,
 we have that
 \[
      \frac{\inner{f_{uu}}{\nu}}{|f_u|^2}
       =-\frac{U_u^2 \inner{f_{U}}{\nu_U}}{U_u^2|f_U|^2}=
         \frac{\inner{f_{UU}}{\nu}}{|f_U|^2}
 \]
 at $p$, which proves the first assertion.
The second assertion follows immediately
from \eqref{eq:knp} and Proposition \ref{prop:old}.
\end{Proof}

\begin{rem}
 Recently, Nu\~no-Ballesteros and the 
 first author \cite{MB}
 defined a notion of
 umbilic curvature $\kappa_u$ for rank one singular points
 in $\R^3$.
 We remark that the unsigned limiting normal curvature
 $\kappa_n:=|\kappa_\nu|$ coincides with $\kappa_u$
 for cuspidal edges,
 see \cite{MS} for details.
\end{rem}

\begin{exa}
 Define a $C^{\infty}$-map  $f\colon{}\R^2\to\R^3$ by
 \[
   f(u,v)= \left(u^4-4u^2v, u^3-3uv,\frac{u^2}{2}-v \right)
            + (u^2-2v)^2(a,b,0)\quad (a,b\in\R).
 \]
 Then the singular set of $f$ is
 $\{v=0\}$ and the origin is  a swallowtail.
 The unit normal vector field of $f$ is given by 
 \[
    \nu(u,v):=
    \frac{\big(3,-8u,-4(3 (-1+a) u^2-8 b u^3-6 a v+16 b u v)\big)}
     {\sqrt{9+64u^2+16(3 (-1+a) u^2-8 b u^3-6 a v+16 b u v)^2}}.
 \]
 Since $f_u(0,0) = 0$ holds,
 by \eqref{eq:kn1},
 the limiting normal curvature at $(0,0)$ is
 $8a$ (cf.\ \eqref{eq:kn1}).
 On the other hand,
 the limiting normal curvature on
 $\hat\gamma(u)=f(u,0)$, $u\ne0$ is
 \[
    \frac{
       \inner{\hat\gamma''(u)}{\nu(u,0)}}
            {|\hat\gamma'(u)|^2}
        =8 a-\frac{64 b}{3}u+O(u^2),
 \]
 where
 $O(u^{\alpha})$ ($\alpha>0$)
 is a term such that 
 $O(u^{\alpha})/|u|^{\alpha}$ is bounded near $u=0$.
\end{exa}

Let $f:\Sigma^2\to (M^3,g)$ be a frontal.
Suppose that $p$ is of the first kind.
Let $\kappa(t)=|\hat \gamma''(t)|$
be the curvature function of $\hat \gamma:=f\circ \gamma$ in $M^3$,
where $t$ is the arclength parameter of $\hat\gamma$.
Then $\kappa_\nu=\inner{\hat \gamma''(t)}{\nu}$ gives the normal
part of $\hat\gamma''(t)$.
On the other hand, we set
\begin{equation}\label{eq:ks}
 \kappa_s(t):=\sgn(d\lambda(\eta))\,
            \det_g\biggl (\hat\gamma'(t),\hat\gamma''(t),\nu(\gamma(t))
                \biggr),
\end{equation}
which is called the {\it singular curvature},
where $\eta(t)$ is a null vector field along $\gamma$
such that $\{\gamma',\eta\}$ is compatible with the orientation of
$\Sigma^2$,
and $\sgn(d\lambda(\eta))$ is the sign of the function
$d\lambda(\eta)$ at $\gamma(t)$.
Since $\kappa_s$ can be considered as the tangential
component of $\hat\gamma''(t)$,
it holds that
\begin{equation}\label{eq:k3}
  \kappa^2=(\kappa_s)^2+(\kappa_\nu)^2.
\end{equation}
When $p$ is a cuspidal edge, the singular curvature is
defined in \cite{SUY}. In this case, it was shown in \cite{SUY}
that $\kappa_s$ depends only on the
first fundamental form (namely, it
is {\it intrinsic\/}) 
and we can prove the following assertion as an application of
\cite{SUY}.
\begin{fact}\label{fact:2}
 Let $p$ be a non-degenerate singular point
 of the first kind of a frontal $f$.
 If the Gaussian curvature function $K$ is bounded
 near $p$, then the limiting normal curvature
 $\kappa_\nu$ vanishes at $p$.
\end{fact}

\begin{Proof}
 In the first paragraph of \cite[Theorem 3.1]{SUY},
 it was stated that the second fundamental form of
 a front $f$
 vanishes at non-degenerate singular points
 if $K$ is bounded.
 In fact, the proof there needed only
 that $f$ is a frontal.
 By \eqref{eq:kn1}, $\kappa_\nu$ is a coefficient
 of the second fundamental form, and must vanish.
\end{Proof}

Fact \ref{fact:2} 
suggests
the existence of a new intrinsic invariant related
to the behavior of $\kappa_\nu$ and the Gaussian curvature.
This is our motivation for introducing the \lq product curvature'
in the following section.
As a consequence of
our following discussions,
the converse assertion of
Fact~\ref{fact:2} 
is obtained in Theorem \ref{thm:main1}.
For a non-degenerate singular point $p$ of the first kind, the
derivatives
of the
limiting normal curvature and the singular curvature with respect to
the arclength parameter $t$ of
$\hat\gamma(t):=f\bigl(\gamma(t)\bigr)$
\[
  \kappa'_\nu(p):=
  \left. \frac{d \kappa_\nu(t)}{dt}\right|_{t=0},
  \quad
  \kappa'_s(p):=
  \left. \frac{d \kappa_s(t)}{dt}\right|_{t=0},
\]
where $\gamma(t)$ is the singular curve such that $p=\gamma(0)$,
are called the {\it derivate limiting normal curvature\/}
and the {\it derivate singular curvature}, respectively,
which will be useful in the following sections.

\section{Cuspidal curvature}
\label{sec:cuspidal}
Let $\sigma(t)$ be a curve in the Euclidean plane $\R^2$
and suppose that $t=0$ is a $3/2$-cusp.
Then the {\it cuspidal curvature\/} of the $3/2$-cusp
is given by (cf.\ \cite{SUY2} and \cite{SU})
\begin{equation}\label{eq:kc0}
 \tau:=\frac{\det(\sigma''(0),\sigma'''(0))}{|\sigma''(0)|^{5/2}}.
\end{equation}
In \cite{SU}, the following formula was shown
\begin{equation}\label{eq:kc02}
 \tau=2\sqrt{2}\lim_{t\to 0}\sqrt{|s(t)|}\kappa(t),~
 \quad s(t):=\int_{0}^t |\sigma'(w)|\,dw,
\end{equation}
where $\kappa(t)$ is the curvature function of
$\sigma(t)$.

Let $p$ be a singular point of the first kind of
a frontal $f\colon{}\Sigma^2\to (M^3,g)$,
and $\gamma(t)$ a singular curve such that $\gamma(0)=p$. 
Let $\eta$
be a non-vanishing vector field defined on a neighborhood of $p$
such that $\eta$ points in the null direction along $\gamma$,
which is called an {\it extended null vector field}. We set
\[
   f_{\eta}:=df(\eta),\quad
   f_{\eta\eta}:=\nabla_{\eta}f_{\eta}, \quad
   f_{\eta\eta\eta}:=\nabla_{\eta}f_{\eta\eta}.
\]
By definition, $f_{\eta}$ vanishes along $\gamma$.
Let $(u,v)$ be an admissible coordinate system at $p$.
Since $\partial_u$ and $\eta$ are linearly independent at $p$,
we may
consider $\lambda$ in Definition~\ref{def:lambda}  as
$\lambda=\det_g(f_u,f_{\eta},\nu)$.
Since $\lambda=0$ along $\gamma$,
the non-degeneracy of $p$ yields that
\begin{equation}
 \label{eq:non-degfvv}
   0\ne \lambda_{\eta}=\det_g{(f_u,f_{\eta\eta},\nu)},
\end{equation}
namely,
$f_u$ and $f_{\eta\eta}$ are linearly independent
at the non-degenerate singular point.
Define the exterior product $\times_g$ so that
\[
  \inner{\vect a \times_g \vect b}{\vect c}
     =\det_g(\vect a, \vect b,\vect c)
\]
holds for each $\vect a, \vect b, \vect c\in T_qM^3$
($q\in M^3$).
Since the singular points are non-degenerate,
the linear map $df:T_{\gamma(t)}\Sigma^2\to T_{\hat\gamma(t)}M^3$ 
is of rank $1$.
Hence $\hat \gamma'(t)$ is proportional to $f_u$, and then
$\hat\gamma'(t)\times_g f_{\eta\eta}\bigl(\gamma(t)\bigr)$
does not vanish, where $\hat \gamma(t):=f\circ \gamma(t)$. 
Then we define the {\it cuspidal curvature\/}  for
singular points of the first kind as
\begin{equation}\label{eq:kc}
 \kappa_c(t):=
  \frac{\bigl|\hat \gamma'(t)\bigr|^{3/2}\,\,%
     \det_g\biggl(
       \hat
       \gamma'(t),f_{\eta\eta}(\gamma(t)),f_{\eta\eta\eta}(\gamma(t))
       \biggr)}{%
       \biggl|\hat \gamma'(t)\times_g f_{\eta\eta}(\gamma(t))
       \biggr|^{5/2}},
\end{equation}
where $\hat \gamma=f(\gamma(t))$ and $\eta$ is chosen so that
$\{\gamma', \eta\}$ consists of
a positively oriented frame on $\Sigma^2$ along $\gamma$.
If $t$ is the arclength parameter of $\hat\gamma$, then
\[
   \kappa'_c(p):=\left. \frac{d \kappa_c(t)}{dt}\right |_{t=0}
\]
is called the {\it derivate cuspidal curvature}.
The two invariants $\kappa_c,\kappa'_c$ do not depend on the choice of
the unit normal vector $\nu$.

The following assertion gives a geometric meaning of
cuspidal curvature.

\begin{prop}\label{prop:geom0}
 Let $f$ be a front in $\R^3$ and $p$ a cuspidal edge,
 and let $\gamma(t)$ be the singular curve
 such that $\gamma(0)=p$.
 Then the intersection of the image of $f$ and a plane $P$
 passing through $f(p)$
 perpendicular to $\hat\gamma'(0)$
 consists of a $3/2$-cusp $\sigma$ in $P$.
 Furthermore, $\kappa_c(0)$ coincides
 with the cuspidal curvature of 
 the plane curve $\sigma$
 {\rm(}cf.\ \eqref{eq:kc0}{\rm)} at $p$.
\end{prop}
\begin{Proof}
 Without loss of generality, 
 we may assume that $p = (0,0)$
 and $f(p)$ is the origin in $\R^3$.
 We denote by $\Gamma$ the
 intersection of the image of $f$
 and the plane $P$ as in the statement.
 Let $(u,v)$ be an admissible coordinate system.
 Then we have that
 \[
    \Gamma=
      \biggl\{f(u,v)\,;\,
       \inner{f(u,v)}{f_u(0,0)}=0
      \biggr\}.
 \]
 Since $f_u(0,0)\ne 0$,
 by applying the implicit  function theorem
 for 
\[
    \inner{f(u,v)}{f_u(0,0)}=0,
\]
 there exists a $C^\infty$-function
 $u=u(v)$ ($u(0)=0$) such that
 \begin{equation}\label{eq:implicit}
  \inner{f\bigl(u(v),v\bigr)}{f_u(0,0)}=0, 
 \end{equation}
 and
 \begin{equation}\label{eq:sigma}
  \sigma(t):=f(u(t),t)
 \end{equation}
 gives a parametrization of
 the set $\Gamma$.
 Differentiating \eqref{eq:implicit},
 we have that
 \begin{equation}\label{eq:implicit2}
  \inner{u'(v)f_u(u(v),v)+f_v(u(v),v)}{f_u(0,0)}=0,
 \end{equation}
 where $u':=du/dv$.
 Since $f_v(0,0)=0$, we have that
 \begin{equation}\label{eq:implicit20}
  u'(0)=0.
 \end{equation}
 Differentiating \eqref{eq:implicit2} again,
 one can get
 \begin{equation}\label{eq:implicit3}
  u''(0)=
   -\frac{\inner{f_{u}(0,0)}{f_{vv}(0,0)}}
   {\inner{f_{u}(0,0)}{f_{u}(0,0)}}.
 \end{equation}
 On the other hand,
 differentiating \eqref{eq:sigma},
 the equation 
 \eqref{eq:implicit20} yields that
 \begin{align}
  \sigma'(0)&=0, \nonumber\\
  \label{eq:16}
  \sigma''(0)&=f_{vv}(0,0)+f_u(0,0)u''(0), \\
  \sigma'''(0)&=
  f_{vvv}(0,0)+3f_{uv}(0,0)u''(0)+f_u(0,0)u'''(0),
  \nonumber
 \end{align}
 which imply that
 \begin{multline*}
  \det\bigl(f_{u}(p),\sigma''(0),\sigma'''(0)\bigr)
  =
  \det \bigl(f_{u}(p),f_{vv}(p),
  f_{vvv}(p)+3f_{uv}(p)u''(0)\bigr)\\
  =
  \det(f_{u}(p),f_{vv}(p),
  f_{vvv}(p))
  +
  3u''(0)\det(f_{u}(p),f_{vv}(p),
  f_{uv}(p)).
 \end{multline*}
 Since $p$ is not a cross cap,
 $\det(f_{u}(p),f_{vv}(p), f_{uv}(p))$ 
 vanishes  by the well-known Whitney's
 criterion for cross caps \cite[Page 161 (b)]{whitney}.
 Thus, we get the identity
 \begin{equation}\label{eq:bunshi}
  \det(f_{u}(p),\sigma''(0),\sigma'''(0))
   =
   \det(f_{u}(p),f_{vv}(p),
   f_{vvv}(p)).
 \end{equation}
 By \eqref{eq:16} and \eqref{eq:implicit3}, we have that
 \begin{equation}\label{eq:bunbo}
  |\sigma''(0)|^2
   =\inner{f_{vv}(p)}{f_{vv}(p)}-
   \frac{\inner{f_{vv}(p)}{f_{u}(p)}^2}
   {\inner{f_{u}(p)}{f_{u}(p)}}
   =
   \frac{|f_{u}(p)\times f_{vv}(p)|^2}
   {|f_{u}(p)|^2}.
 \end{equation}
 By \eqref{eq:kc0}, \eqref{eq:kc},
 \eqref{eq:bunshi} and
 \eqref{eq:bunbo}, we get the assertion.
\end{Proof}
\begin{rem}\label{rem:MS}
 In \cite{MS}, a normal form of a cuspidal edge
 singular point in $\R^3$ was given, and
 its proof can be applied to a given
 singular point $p$
 of the first kind without any modification.
 So there exist an isometry $T$  of $\R^3$ and 
 a local coordinate system $(u,v)$
 centered at  $p$ such that
 \[
  T\circ f(u,v)=\left(
         u,\frac{a(u)u^2+v^2}2,
         \frac{b_0(u)u^2+b_2(u)uv^2}2+\frac{b_3(u,v)v^3}6
    \right),
 \]
 where $a$, $b_0$, $b_2$ and $b_3$ are $C^\infty$-functions.
 It holds that
 \[
   \kappa_s(p)=a(0),\quad
   \kappa_\nu(p)=b_0(0),\quad
   \kappa_c(p)=b_3(0,0).
 \]
 Moreover, we have that
 \begin{align*}
  \kappa'_s(p)&=b_0(0)b_2(0)+3a'(0), \\
  \kappa'_\nu(p)&=-a(0)b_2(0)+3b'_0(0), \quad
  \kappa'_c(p)=(b_3)_u(0,0).
 \end{align*}
 In particular, since $\kappa_s$ is intrinsic,
 so is $\kappa'_s$, which gives
 a geometric meaning for the coefficient $a'(0)$.
 On the other hand,
 \[
    k_t(p):=b_2(0),
      \qquad
    k_i(p):=3b'_{0}(0),
 \]
 are called
 the {\it cusp-directional torsion\/}
 and
 the {\it edge inflectional curvature}, which were
 investigated in \cite{MS}.
\end{rem}

\begin{figure}[htbp]
\centering
 \begin{tabular}{c@{\hspace{8em}}c}
\raisebox{3ex}{
  \includegraphics[width=.25\linewidth]{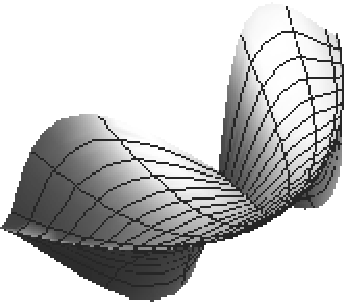}}&
  \includegraphics[width=.22\linewidth]{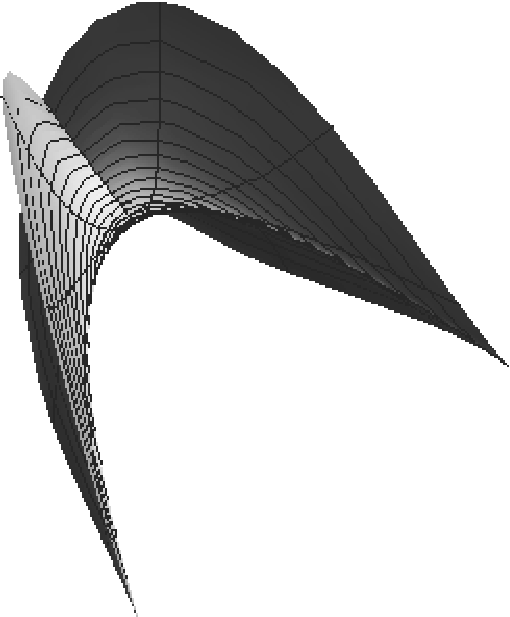}
 \end{tabular}
 \caption{Cuspidal cross caps of $\kappa_s>0$ and of $\kappa_s<0$.}
 \label{fig:ccr}
\end{figure}

\begin{exa}\label{exa:new}
 Let $\kappa_{\nu}$, $\kappa_c$ and $c$ ($>0$) be constants, 
 and set
 \[
   f(u,v):=
     \left(u,\frac{\kappa_s u^2}{2}+\frac{v^2}{2},
     \frac{c u v^3}{6} +
    \frac{\kappa_\nu u^2}{2}
     \right).
 \]
 Then the $u$-axis is the singular curve, and
 $\partial/\partial v$ gives the null direction.
 A unit normal vector of $f$ is given by
 \[
    \nu:=\frac1{\delta}
           \left(3 c \kappa_s u^2 v-c v^3-6 \kappa_\nu u,
                -3 c    u v,
             6\right),
 \]
 where
 \[
    \delta:=
      \sqrt{9 c^2 u^2 v^2+\left(c v \left(v^2-3 \kappa_s u^2\right)+
         6 \kappa_\nu u\right)^2+36}.
 \]
 By \ref{item:Fact3} of Fact~\ref{fact:criteria},
 $f(u,v)$ has a cuspidal cross cap
 singular point at $(0,0)$.
 The singular curvature 
 and the limiting normal curvature
 of $f$ at $(0,0)$ are equal to $\kappa_s$
 and $\kappa_\nu$, respectively.
 Fig.\ \ref{fig:ccr} left (resp.\ right) is the
 case with $c=6,\kappa_{\nu}=0$ and  $\kappa_s=2$ 
 (resp.\ $\kappa_s=-2$). The cuspidal curvature along the $u$-axis is 
 \[
  \kappa_c(u)=
   \frac{cu\bigl(1+\left((\kappa_{\nu})^2+(\kappa_s)^2\bigr) u^2\right)^{3/4}}{%
                (1+(\kappa_{\nu})^2 u^2)^{5/4}}.
 \]
\end{exa}

Let $p$ be a point of $\Sigma^2$. 
One can define the \lq blow up\rq\
of the manifold $\Sigma^2$ at $p$, that is,
there exist a $C^\infty$-manifold
$\hat \Sigma^2_p$ and consider the blowing up
$\Phi:\hat \Sigma^2_p\to \Sigma^2$
as in the appendix.

\begin{defi}\label{def:rational}
 Let $\U\,(\subset \Sigma^2)$ be a   
 neighborhood of $p$, and
 $\O$ an open dense subset of $\U$.
 A real-valued $C^\infty$-function $\phi$ defined on 
 $\O$ is called {\it rationally bounded\/}
 at $p$
 if 
 there exists a real-valued $C^\infty$-function
 $\lambda$
 defined on $\Phi^{-1}(\U)(\subset \hat \Sigma^2_p)$ 
 such that
 \begingroup
 \renewcommand{\theenumi}{{\rm(\alph{enumi})}}
 \renewcommand{\labelenumi}{{\rm(\alph{enumi})}}
 \begin{enumerate}
  \item\label{item:rat-a}
       $\Phi^{-1}(\O\setminus\{p\})=
       \Phi^{-1}(\U\setminus \{p\})\cap \lambda^{-1}(\R\setminus\{0\})$,
  \item\label{item:rat-b}
       $\lambda^{-1}(\{0\})\cap \Phi^{-1}(\{p\})$ 
       is a finite set, and 
  \item\label{item:rat-c}
       there exists a $C^\infty$-function
       $\psi:\Phi^{-1}(\U)\to \R$ such that
       $\psi(q)=\lambda(q) \,\phi(\Phi(q))$ holds
       for $q\in \Phi^{-1}(\O)$.
 \end{enumerate}
 \endgroup
 Moreover, if  there exists a constant $c$ such that
\begin{equation}\label{eq:psi-lambda}
\psi(q)=c \lambda(q)\qquad (q\in \Phi^{-1}(\{p\})),
\end{equation}
then the function $\varphi$ 
is called {\it rationally continuous\/}
at $p$. 
\end{defi}

By definition, the continuity implies the rational continuity,
and the rational continuity implies the rational boundedness.
If $\phi$ is not continuous but rationally continuous 
or rationally bounded at $p$, then $p\not \in \O$.
Let $(\U;u,v)$ be a local coordinate system
centered at $p$.
We set
\[
   u:=r \cos \theta,\quad v:=r \sin \theta
\qquad (r\ge 0).
\]
Suppose that 
a function $\phi(u,v)$ is rationally bounded.
By definition,
there exists a real-valued $C^\infty$-function
$\lambda$ satisfying 
\ref{item:rat-a}, \ref{item:rat-b} and \ref{item:rat-c}.
We may consider $\lambda$ the function
of two variables $r,\theta$. 
We denote by $\theta_1,\dots,\theta_n\in [0,2\pi)$
the zeros of the function $\lambda(0,\theta)$.
It can be easily checked that
for each sufficiently small
$\varepsilon>0$,
there exist positive constants $\delta,C$
such that the inequality
\[
   |\phi(r \cos\theta,r \sin\theta)|<C
          \qquad (0<r<\delta)
\]
holds whenever $|e^{i\theta}-e^{i\theta_j}|>\varepsilon$
for each $j=1,\dots,n$, where $i=\sqrt{-1}$.
This is a reason why we say that $\phi$
is rationally bounded.
Moreover, if the 
function $\phi(u,v)$ is rationally continuous,
one can easily check that
\[
   \lim_{\substack{(u,v)\to (0,0)\\
   (u,v)\in \O}}
   \phi(u,v)=c
\]
holds, where $c$ is the constant 
as in \eqref{eq:psi-lambda}.
So, we say that $\phi$
is of rationally continuous.

\begin{exa}
 The functions
 \[
   \phi_1(u,v):=1+\frac{uv}{u^2-v^2},
   \qquad
   \phi_2(u,v):=1+\frac{u^2v}{u^2-v^2}
 \]
 are rationally bounded at $(0,0)$
 by setting 
 $\lambda(r,\theta):=\cos^2\theta-\sin^2\theta$.
 In this case, 
 $\lambda(0,\theta)=0$ if and only if
 $\theta=(2j+1)\pi/4$ ($j=0,1,2,3$).
 Moreover, the function
 $\phi_2$ is rationally continuous at $(0,0)$
 by setting 
 $\psi:=\lambda+r\cos^2\theta\sin \theta$.
\end{exa}
\begin{exa}
 For $a\in\R\setminus\{0\}$ and $k\geq 2$, we set
 $f(u,v):=(u,v^2,v^3+a u^k)$.
 Then its Gaussian curvature function  is given by
 \[
     K(u,v)=\frac{3ak(k-1)}{4v}\biggl(u^{k-2}+O(1)\biggr),
 \]
 where $O(1)$ is a smooth function of $u$ and $v$.
 Hence $K$ is rationally bounded (resp.\ rationally continuous) 
 at the origin (with $\lambda(r,\theta)=\sin\theta$)
 if $k\ge 3$ (resp.\ $k\geq 4$).
\end{exa}

\begin{defi}\label{def:adapted}
 An admissible coordinate system $(u,v)$ 
 (cf.\ Definition~\ref{def:admissible})
 at
 a singular point
 $p$ of the first kind of a frontal $f$
 is called {\it adapted\/}  
 if
 it is compatible with respect to
 the orientation of $\Sigma^2$,
 and the following conditions hold
 along the $u$-axis
 \begingroup
 \renewcommand{\theenumi}{(\alph{enumi})}
 \renewcommand{\labelenumi}{(\alph{enumi})}
 \begin{enumerate}
  \item\label{item:adapted:1}
       $|f_u| = 1$,
  \item\label{item:adapted:2}
       $f_v=0$, in particular, the singular set coincides with
       the $u$-axis,
  \item\label{item:adapted:3}
       $\{f_u, f_{vv},\nu\}$ is an
       orthonormal frame.
 \end{enumerate}
 \endgroup
\end{defi}
If $(U,V)$ is another adapted coordinate system at $p$,
then it holds that 
\begin{equation}\label{eq:important}
 V_{v}(p)=1.
\end{equation}
The existence of an adapted coordinate system was
shown in \cite[Lemma 3.2]{SUY}.
(In  fact, the proof of \cite[Lemma 3.2]{SUY} does not
use the assumption that $f$ is a front.)
Since
$\inner{f_{vv}}{\nu}=-\inner{f_{v}}{\nu_v}=0$
along the $u$-axis, \ref{item:adapted:3}
is equivalent to the condition
that $\inner{f_u}{f_{vv}}=\inner{f_u}{f_v}_v=0$
and $\inner{f_{vv}}{f_{vv}}
=\frac{1}{2}\inner{f_{v}}{f_{v}}_{vv}=1$
along the $u$-axis.
Thus, $(u,v)$ is adapted if and only if
\begin{equation}\label{eq:I}
  \inner{f_u}{f_u}=1,\qquad
  \inner{f_v}{f_v}=\inner{f_u}{f_v}_v=0,\qquad
  \inner{f_v}{f_v}_{vv}=2 
\end{equation}
hold along the $u$-axis.
In particular, the adapted coordinate system
can be characterized in terms of the first
fundamental form.

We fix an adapted coordinate system $(\U;u,v)$
of a frontal $f\colon{}\Sigma^2\to M^3$
such that 
$\{f_u, f_{vv},\nu\}$ is a frame compatible with respect to
the orientation of $M^3$.
Then $f_1:=f_u$ and $f_2:=f_{v}/v$ give smooth
vector fields on $\U$.
By definition, $f_2(u,0)=f_{vv}(u,0)$.
We now set
\begin{equation}\label{eq:hatG}
  \hat g_{ij}^{ }:=\inner{f_i}{f_j},\quad
  \hat h_{ij}:=-\inner{f_{i}}{\nu_j} \qquad (i,j=1,2),
\end{equation}
where $\nu_1:=\nu_u$ and $\nu_2:=\nu_v$.
By \ref{item:adapted:1} 
in Definition~\ref{def:adapted},
we have 
\begin{equation}\label{eq:g1122}
\hat g_{11}(u,0)=1.
\end{equation}
On the other hand, we have
\begin{equation}\label{eq:g12}
  \hat g_{12}(u,0)=\inner{f_u(u,0)}{f_{vv}(u,0)}=0
\end{equation}
and
\begin{equation}\label{eq:g220}
  \hat g_{22}(u,0)=
\lim_{v\to 0}\inner{f_v(u,v)/v}{f_{v}(u,v)/v}=
\inner{f_{vv}(u,0)}{f_{vv}(u,0)}=1.
\end{equation}
Then the mean curvature function $H$ of $f$
is given by
\[
   H=\frac{\hat g_{11}^{ }\hat h_{22}
       -2v \hat g_{12}^{ }\hat h_{21}+v\hat g_{22}^{ } \hat h_{11}}{2v(
     \hat g_{11}^{ }\hat g_{22}^{ }-(\hat g_{12}^{ })^2)},
\]
namely $\hat H:=vH$ is a $C^\infty$-function
of $u,v$ such that
$\hat H=\frac{1}{2}\hat h_{22} +O(v)$,
where $O(v)$ is a term
such that $O(v)/v$ is bounded near $v=0$.
Differentiating $f_v=v k$ ($k:=f_2$),
we have that
\[
   f_{vv}=k+v k_v,\quad
   f_{vvv}=2k_v+v k_{vv}.
\]
By \ref{item:adapted:3},
$\nu=f_u\times_g f_{vv}$
holds at $(u,0)$.
Then we have
\begin{equation}\label{eq:kc-h}
  \kappa_c 
    =\det_{g}(f_u,f_{vv},f_{vvv})=
                     \inner{\nu}{f_{vvv}}
    =2\inner{\nu}{k_{v}}
        =-2\inner{\nu_v}{k}=2 \hat h_{22}
\end{equation}
along the $u$-axis.
So it holds that
\begin{equation}\label{eq:hatH}
   4\hat H(u,0)=\kappa_c(u)=
                \kappa_c(p)+\kappa'_c(p)u+O(u^2),
\end{equation}
which yields the following assertion.
\begin{prop}\label{prop:mean}
 Let $p$ be a singular point of the first
 kind of a frontal $f$.
 Then  the mean curvature function $H$ is
 bounded at $p$
 if and only if  $\kappa_c$ vanishes on a neighborhood
 of $p$ in the singular set
 {\rm(}cf.\ {\rm \cite[Corollary 3.5]{SUY}}{\rm)}.
 Moreover, $H$ is rationally bounded {\rm (}resp.\ rationally continuous{\rm)}
at $p$ if and only if $\kappa_c(p)=0$ {\rm(}resp.\
 $\kappa_c(p)=\kappa'_c(p)=0${\rm)}.
\end{prop}
We next discuss the Gaussian curvature $K$ of $f$.
The Gaussian curvature $K$ is expressed as
\begin{equation}\label{eq:cg}
  K=K_{\ext}+c_{g}(\nu^\perp),
\end{equation}
where $K_{\ext}$ is the extrinsic Gaussian
curvature (i.e.\ the determinant of the shape operator), 
and $c_{g}(\nu^\perp)$ is the
sectional curvature of the metric $g$
with respect to the $2$-subspace perpendicular
to $\nu$ in $T_{f(p)}M^3$ (cf.\ \cite[Proposition 4.5]{KN}).
Since $c_{g}(\nu^\perp)$ is bounded at $p$, we have that
\begin{equation}\label{eq:hatK}
  K=K_{\ext}+O(1)=
     \frac{\hat h_{11}\hat h_{22}-
         v\bigl(\hat  h_{21}^{})^2}{v(\hat g_{11}\hat g_{22}-
    (\hat g_{12}^{})^2\bigr)}+O(1),
\end{equation}
where $O(1)$ is a smooth function of $u$
and $v$.

In particular, $\hat K:=v K$ is a $C^\infty$-function
of $u$, $v$.
By \ref{item:adapted:1} of Definition~\ref{def:adapted},
$u$ is the arclength parameter of the 
image of the singular curve $\hat \gamma(u)=f(u,0)$.
Then we have
\begin{equation}\label{eq:knu}
  \kappa_{\nu}(u)=\inner{f_{uu}(u,0)}{\nu(u,0)}
                 = -\inner{f_u(u,0)}{\nu_u(u,0)}
		 =\hat h_{11}(u,0).
\end{equation}
By 
\eqref{eq:g1122}, \eqref{eq:g12}, \eqref{eq:g220}, 
\eqref{eq:kc-h}, \eqref{eq:knu} and \eqref{eq:hatK},
we have that
\begin{equation}\label{eq:k-expand}
   2\hat K(u,0)=\kappa_{\Pi}(u)=\kappa_\Pi^{}(p)+
            \kappa'_\Pi(p)u+O(u^2),
\end{equation}
where
\begin{equation}\label{eq:prod}
 \kappa_{\Pi}^{}(u): =\kappa_{\nu}(u)\kappa_c(u),\qquad
 \kappa'_\Pi(p):=\kappa'_\nu(p)\kappa_c(p)+\kappa_\nu(p)\kappa'_c(p).
\end{equation}
We call $\kappa_\Pi^{}$ and $\kappa'_\Pi$ the {\it product curvature\/}
and the {\it derivate product curvature}, respectively.
By \ref{item:adapted:1} of Definition \ref{def:adapted} and 
\eqref{eq:important},
$\dy\lim_{v\to 0}vK(u,v)$ does not depend on the choice of
adapted coordinate system.
Since the adapted coordinate system
can be characterized in terms of the first fundamental form,
we get the following:

\begin{thm}\label{thm:main1}
Let $f$ be a frontal and $p$ a singular point
 of the first kind.
Then  $\kappa_\Pi^{}$ and $\kappa'_\Pi$
are both intrinsic invariants%
\footnote{
 This might be considered as a variant of Gauss' Theorema Egregium.
}.
 Moreover, the Gaussian curvature
 $K$ is rationally bounded 
 {\rm(}resp.\ rationally continuous{\rm)}
 at $p$
 if and only if
 $\kappa_\Pi^{}=0$ {\rm(}resp.\  
$\kappa_\Pi^{}=\kappa'_\Pi=0${\rm)}
 holds at $p$.
 Furthermore, $K$ is bounded on a neighborhood $\U$
 of $p$ if and only if $\kappa_\Pi^{}=0$ along
 the singular curve in $\U$.
\end{thm}

\begin{Proof}
 The rational boundedness
 (resp.\ rational continuity)
 of $K$ can be proved by setting $\lambda=\sin \theta$,
 since $\hat K=v K$ and $v=r \sin \theta$.
 The last assertion (i.e.\ the boundedness of $K$ on $\U$)
 is proved as follows:
 If $K$ is bounded on a neighborhood $\U$
 of $p$, then $\kappa_{\Pi}$ must vanish
 identically because of the identity
 $2\hat K(u,0)=\kappa_{\Pi}(u)$ (cf.\ \eqref{eq:k-expand}).
 On the other hand, we suppose $\kappa_{\Pi}(u)=0$
 along the $u$-axis.
 Then by \eqref{eq:k-expand}, $\hat K(u,0)=0$ and hence
 there exists a $C^\infty$-function $\phi(u,v)$
 defined on a neighborhood $p$ such that 
 \[
   (v K(u,v)=)\hat K(u,v)=v \phi(u,v).
 \]
 Thus the identity $K(u,v)=\phi(u,v)$
 holds whenever $v\ne 0$.
 This implies that the Gaussian curvature
 is bounded.
\end{Proof}

\begin{rem}
 In \cite{SUY}, a necessary and sufficient condition
 of the boundedness of Gaussian curvature at
 non-degenerate front singularities is given.
 However, for singular points on a surface which 
 is frontal but not front,
 the criterion in  \cite{SUY} does not work.
 In this sense,
 the last statement of Theorem \ref{thm:main1}
 is a generalization of \cite[Theorem 3.1]{SUY}.
\end{rem}

By \cite[Corollary 3.5]{SUY}, $\kappa_c(p)\ne 0$
holds if $p$ is a cuspidal edge.
In fact, the following assertion holds.

\begin{prop}
\label{lem:kappac-psiccr}
 Let $f$ be a frontal and $p$ a non-degenerate
 singular point  of the first kind.
 Then $f$ is a front at $p$ if and only if
 $\kappa_c(p)\ne 0$. 
\end{prop}

\begin{Proof}
 On an adapted coordinate system,
 $\kappa_c=\det_g(f_u,f_{vv},f_{vvv})$
 holds at $(u,0)$.
 On the other hand,
 since $f_v(u,0)=0$ holds, $f_{uv}(u,0)=0$.
 By using the formula
 \[
   \det_g(\vect x,\vect a\times \vect x,\vect b\times \vect x)
       =|\vect x|^2 \det_g(\vect x,\vect a,\vect b)
      \qquad (\vect a,\vect b,\vect x\in T_{o}M^3,\,\,o\in M^3),
 \]
 it holds that
 \begin{align*}
   \psi_{\ccr}(u)&=
\det_g(f_u,\nu,\nu_v)
=
\left.    \det_g(f_u,f_u\times_g f_{vv},f_u\times_g f_{vvv})\right|_{(u,v)=(u,0)}\\
 &=
    \left. \det_g(f_u,f_{vv},f_{vvv})\right|_{(u,v)=(u,0)}.
 \end{align*}
 Hence the assertion follows from Lemma \ref{lem:psiccr}.
\end{Proof}
Proposition \ref{lem:kappac-psiccr} yields the following:
\begin{cor}\label{cor:cusp}
 Let $p$ be a cuspidal edge.
 Then $K$ is rationally bounded
 {\rm(}resp.\ rationally continuous{\rm)}
 at $p$
 if and only if
 $\kappa_\nu=0$ {\rm(}resp.\
 $\kappa_\nu=\kappa'_\nu=0${\rm)}
 holds at $p$.
\end{cor}

\begin{cor}\label{cor:new-cusps}
 Let $p$ be a singular point of the
 first kind of a frontal. If the mean curvature
 $H$ is bounded {\rm(}resp.\ rationally bounded,
 rationally continuous{\rm)} at $p$,
 then the Gaussian curvature $K$ is bounded
 {\rm(}resp.\ rationally bounded, rationally continuous{\rm)}
 at $p$.
\end{cor}

\begin{Proof}
 The assertion follows from
 Proposition \ref{prop:mean}, 
 the identity \eqref{eq:prod}
 and Theorem \ref{thm:main1}.
\end{Proof}

A singular point $p\in \Sigma^2$ of a map $f:\Sigma^2\to M^3$
is a {\em $5/2$-cuspidal edge}\/
if the map germ $f$ at $p$ is right-left equivalent to
$(u,v)\mapsto(u,v^2,v^5)$ at the origin.
The map $f$ is a frontal on a neighborhood of the $5/2$-cuspidal edge
$p$, but not a front at $p$.

\begin{cor}\label{cor:52cusp}
 Let $p$ be a $5/2$-cuspidal edge.
 Then the mean curvature $H$ and
 the Gaussian curvature $K$ are both bounded 
 near $p$.
\end{cor}
\begin{Proof}
 Since singular points on a neighborhood of $p$ consists 
 of $5/2$-cuspidal edges,
 all singular points are not front singularities.
 Then the cuspidal curvature $\kappa_c$
 vanishes identically on the singular set 
 because of Lemma~\ref{lem:psiccr} 
 and the proof of
 Proposition~\ref{lem:kappac-psiccr}.
 Thus the boundedness of $H$ and $K$ follow
 from Proposition \ref{prop:mean} and
 Corollary \ref{cor:new-cusps}.
\end{Proof}

\begin{exa}
 A useful criterion for $5/2$-cuspidal edge
 singularities is given in \cite{HKS}.
 Applying this,
 one can check that
 the map $f\colon{}\R^2\to\R^3$ 
 defined by
 \[
    f(u,v):=(u,a u^2+v^2,c u^2+b v^5)
 \]
 has $5/2$-cuspidal edge singularities
 along the $u$-axis.
 The unit normal vector field
 is given by
 \[
  \nu:=
   \frac{1}{\sqrt{\left(4 c u-10 a b u v^3\right)^2+25 b^2 v^6+4}}
   \left(10 a b u v^3-4 c u,-5 b v^3,2\right).
 \]
 The  limiting normal curvature
 are given by
 \begin{align*}
  \kappa_{\nu}(u)&
  =\frac{2 c}{\sqrt{4 c^2 u^2+1} \left(4 a^2 u^2+4 c^2 u^2+1\right)}.
 \end{align*}
 On the other hand, the Gaussian curvature is given by
 \[
   K=-\frac{60 b v \left(5 a b v^3-2 c\right)}
   {\left(25 b^2 v^6 \left(4 a^2 u^2+1\right)
             -80 a b c u^2 v^3+16 c^2 u^2+4\right)^2}
 \]
 which is bounded at the singular set as shown in
 Corollary~\ref{cor:52cusp}.
 In the case of cuspidal edges, the Gaussian curvature
 $K$ is bounded if and only if
 $\kappa_\nu$ vanishes identically.
 However, for this $f$ with $c\ne 0$,
 $K$ is bounded even if 
 $\kappa_\nu\ne 0$.
\end{exa}
\begin{rem}
 In \cite{SUY}, it was pointed out that
 the Gaussian curvature of a front $f$
 takes opposite signs
 on the left and right hand sides of
 the singular curve if $\kappa_\nu\ne 0$.
 This follows immediately from
 the formula 
 \[
   v K=\frac{1}{2}\kappa_\Pi^{}+O(\sqrt{u^2+v^2}),
 \]
 where $O(\sqrt{u^2+v^2})$ is 
 the term such that 
 $O(\sqrt{u^2+v^2})/\sqrt{u^2+v^2}$
 is bounded near $(u,v)=(0,0)$.
\end{rem}

By regarding that the tensor fields $\det_g$ and
$\inner{~}{~}$ are parallel with respect
to the Levi-Civita connection of $(M^3,g)$,
Proposition \ref{lem:kappac-psiccr} and
Fact \ref{fact:criteria} \ref{item:Fact3}
yield that
$\kappa_c(p)=0$ and $\kappa'_c(p)\ne 0$ hold
if $p$ is a cuspidal cross cap.
Then by \eqref{eq:k-expand} and \eqref{eq:prod}, we have
\begin{cor}
 Let $p$ be a cuspidal cross cap.
 Then $K$ is rationally bounded at $p$. Moreover it
 is rationally continuous
 at $p$ if and only if
 $\kappa_\nu(p)= 0$.
\end{cor}

The Gaussian curvature of the
following cuspidal cross cap
is rationally bounded
but $\kappa_\nu\ne0$ holds.
\begin{exa}\label{ex:crcp}
 Let us consider a map $f\colon{}\R^2\to\R^3$
 defined by
 \[
   f(u,v):=
       (u, v^2, u v^3 + u^2).
 \]
 Then
 \[
     \nu(u,v):=\frac{1}{\sqrt{4+4 (2 u+v^3)^2+9 u^2 v^2}}
               \Big(-2 (2 u+v^3), -3uv, 2\Big)
 \]
 gives a unit normal vector field.
 By a direct calculation, one can see that
 $(0,0)$ is a cuspidal cross cap singularity
 with
 $\kappa_\nu(0,0)=2$,
 and the Gaussian curvature $K$ is
 \[
    K(u,v)=
     \frac{12(2 u - 3 v^3)}
           {v \Big(4 + 16 u^2+9u^2v^2+16uv^3+4 v^6\Big)^2}.
 \]
 This is rationally bounded at $(0,0)$.
\end{exa}

\begin{rem}
 If the ambient space is $\R^3$,
 we can take the normal form 
 at a cuspidal edge  $p=(0,0)$
 as in Remark \ref{rem:MS}.
 Then the $C^\infty$-function $\hat K:=v K(u,v)$ satisfies
 \[
  2 \hat K=\kappa_\Pi^{}+
  \kappa'_\Pi u
   -
  v\left(2(b_2)^2+\frac{\kappa_s(\kappa_c)^2}2
    -\frac{8(b_3)_v\kappa_\nu}3
  \right)
  +O(u^2+v^2),
 \]
 where $O(u^2+v^2)$ is 
 the term such that $O(u^2+v^2)/(u^2+v^2)$
 is bounded near $p=(0,0)$.
 If $K\ge 0$ near $p$,
 then $\kappa_\nu=0$ and thus
 \[
    0\le 4 K(0,0)=-4(b_2)^2-\kappa_s(\kappa_c)^2
 \]
 holds. So we have  $\kappa_s\le 0$, which
 reproves the second assertion of \cite[Theorem 3.1]{SUY}
 in the special case that the ambient space is $\R^3$.
\end{rem}

It is classically known that
regular surfaces in $\R^3$ admit
non-trivial isometric deformations, and
it might be interesting to consider
 the existence of such deformations
at cuspidal edge singularities:
Let $\xi(s)$ ($|s|<1$) be a regular
curve on the unit sphere
$S^2(\subset \R^3)$
and let $a(s)$ ($|s|<1$)  be
with the arclength parameter $s$,
a positive valued function.
Then
\[
  f_{a,\xi}(u,v):=
   v \xi(u)+\int_0^u a(s)\xi(s)\,ds
\]
gives a developable surface 
with
singularities on the $u$-axis.
Moreover, \eqref{eq:kc}
yields
$\kappa_c=-2 \mu_g/\sqrt{a}$,
where $\mu_g$ is the geodesic curvature of the spherical curve
$\xi$.
As pointed out in \cite{HHNUY},
moving $\xi$ so that $\mu_g$ varies,
we get an isometric deformation of $f_{a,\xi}$
so that $\kappa_c$ changes.
Thus $\kappa_c$ is not an intrinsic invariant.
It should be remarked that
we cannot conclude that $\kappa_\nu$
is an extrinsic invariant
since
$f_{a,\xi}$ preserves
the limiting normal curvature $\kappa_\nu$
to be identically zero.
However,
using the fact that the product curvature
$\kappa_\Pi^{}$ is intrinsic 
(cf.\ Theorem \ref{thm:main1}),
the existence of isometric deformations
of cuspidal edges
in $\R^3$ which change $\kappa_\nu$
is shown in \cite{NUY}.
See also Teramoto \cite{tera} 
for other geometric properties of cuspidal 
edges and their parallel surfaces. 

\section{Singularities of the second kind.}
\label{sec:general}

We fix a frontal $f:\Sigma^2\to (M^3,g)$ 
in an oriented Riemannian $3$-manifold.
We consider
singular points of the second kind
of $f$.
Typical such singular points are swallowtails.
In this section, we newly define 
\lq normalized singular curvature\rq\ 
$\mu_c$ and \lq normalized
product curvature\rq\ $\mu^{}_{\Pi}$ 
at singular points of the second kind.
Also, 
\lq limiting singular curvature\rq\ $\tau_s$ 
and \lq limiting
cuspidal curvature\rq\ $\tau_c$ 
are defined for
swallowtail singularities.

\begin{defi}\label{def:adapted2}
 Let $p\in \Sigma^2$ be a  singular point
 of the second kind.
 A local coordinate system $(\U;u,v)$ at $p=(0,0)$
 is called {\it adapted\/} at $p$ if
 it is compatible with respect to
 the orientation of $\Sigma^2$,
 and the following conditions hold:
 \begingroup
 \renewcommand{\theenumi}{(\roman{enumi})}
 \renewcommand{\labelenumi}{(\roman{enumi})}
 \begin{enumerate}
 \item\label{item:adapted2:1}
      $f_u(p)=0$,
 \item\label{item:adapted2:2}
      the singular set of $f$ on $\U$
      coincides with the $u$-axis,
 \item\label{item:adapted2:3}
      $|f_v(p)|=1$.
 \end{enumerate}
 \endgroup
\end{defi}

The existence of adapted coordinate system
can be proved easily.
Let $(U,V)$ be another adapted coordinate system,
then the condition $f_u(p)=0$ and
\ref{item:adapted2:3} yield that
\begin{equation}\label{eq:Vv}
 V_v(p)=1 .
\end{equation}

We fix an adapted coordinate system $(u,v)$ at $p=(0,0)$.
Then one can take a null vector field along the $u$-axis in the form
\begin{equation}\label{eq:second-eta}
\eta=\partial_u+\varepsilon(u) \partial_v\qquad \bigl(\varepsilon(0)=0\bigr)
\end{equation}
for a $C^{\infty}$-function $\varepsilon=\varepsilon(u)$.
We can extend this $\eta$ as a vector field defined on
a neighborhood of the origin.
Since $f_\eta=f_u+\varepsilon(u) f_v$ vanishes on
the $u$-axis, there exists a $C^\infty$-function
$\psi$
such that
\begin{equation}\label{eq:psi-def}
   f_{\eta}(u,v) = f_u(u,v) + \varepsilon(u)f_v(u,v) = v \psi(u,v).
\end{equation}
Differentiating this by $v$, we have that
$\psi(0,0)=f_{uv}(0,0)$.
Since $\lambda(u,0)=0$,
the non-degeneracy of $p$ yields that
\[
  0\ne \lambda_v=\det_g(f_{uv},f_v,\nu)
           =\det_g(\psi,f_v,\nu)
\]
at $p$, which implies that 
$\psi(0,0)$ and $f_v(0,0)$ are linearly independent.
We now set
\[
   g_{ij}^{ }:=\inner{f_{u_i}}{f_{u_j}},\quad
   h_{ij}:=-\inner{f_{u_i}}{\nu_{u_j}} \qquad (i,j=1,2),
\]
where $u_1=u$ and $u_2=v$.
Since $f_u=v\psi -\varepsilon f_v$ as in \eqref{eq:psi-def},
it holds that
\begin{equation}\label{eq:gij}
\begin{aligned}
   g_{11}^{ }&=\inner{v\psi-\varepsilon f_v}{v\psi-\varepsilon f_v}
         =v^2|\psi|^2 -2 v \varepsilon\inner{\psi}{f_v}+
           \varepsilon^2 |f_v|^2,\\
   g_{12}^{ }&=\inner{v\psi-\varepsilon f_v}{f_v}
            = v\inner{\psi}{f_v}-\varepsilon|f_v|^2,\\
   g_{22} &=|f_v|^2,
\end{aligned}
\end{equation}
which yields that
\begin{equation}\label{eq:g-det}
  g_{11}^{ }g_{22}^{ }-(g_{12}^{ })^2=v^2(|\psi|^2|f_v|^2-\inner{\psi}{f_v}^2)=
     v^2|\psi\times_g f_v|^2.
\end{equation}
On the other hand,
\[
    \inner{f_u+\varepsilon f_v}{\nu_{u_i}}
          =\inner{f_\eta}{\nu_{u_i}}
          =v\inner{\psi}{\nu_{u_i}} \quad (i=1,2)
\]
holds, namely, we have
\[
     -h_{11}-\varepsilon h_{12}=v\inner{\psi}{\nu_u}, \quad
     -h_{12}-\varepsilon h_{22}=v\inner{\psi}{\nu_v}.
\]
So we have that
\[
 h_{12}=-v\inner{\psi}{\nu_v}-\varepsilon h_{22},\qquad
 h_{11}=-v\inner{\psi}{\nu_u}+v\varepsilon
     \inner{\psi}{\nu_v}+\varepsilon^2 h_{22},
\]
and then
\begin{align*}
    g_{11}h_{22}-2 g_{12}h_{12}&+ g_{22}h_{11}\\
        &= -v|f_v|^2\inner{\psi}{\nu_u+\varepsilon\nu_v}
            +v^2\bigl(|\psi|^2h_{22}+
             2\inner{\psi}{f_v}\inner{\psi}{\nu_v}\bigr)\\
        &= -v|f_v|^2\inner{\psi}{\nu_{\eta}}
            +v^2\bigl(|\psi|^2h_{22}+
             2\inner{\psi}{f_v}\inner{\psi}{\nu_v}\bigr).
\end{align*}
So by \eqref{eq:g-det},
the mean curvature $H$ of $f$ is expressed as
\begin{align*}
 2v H
  &=v\left(\frac{
        g_{11}^{ }h_{22}-2g_{12}^{ }h_{12}+g_{22}^{ }h_{11}}{
        g_{11}^{ }g_{22}^{ }-(g_{12}^{ })^2}\right)\\
  &=\frac{-|f_v|^2\inner{\psi}{\nu_\eta}+
      v\bigl(|\psi|^2 h_{22}
       +2 \inner{\psi}{f_v}\inner{\psi}{\nu_v}\bigr)}
        {|\psi\times_g f_v|^2}.
\end{align*}
Then
$\hat H:=vH$ is a $C^\infty$-function of $u$, $v$.
We define two constants $\mu_c(p)$ and $\mu'_c(p)$
by the expansion 
\begin{equation}\label{eq:hathexpand}
  2\hat H(u,0)
  =\frac{-|f_v(u,0)|^2\inner{\psi(u,0)}{\nu_\eta(u,0)}}
        {|\psi(u,0)\times_g f_v(u,0)|^2}
  =\mu_c(p)+\mu'_c(p) u +O(u^2).
\end{equation}
Thus, $H$ is rationally bounded (resp.\
rationally continuous) at $p$ if and only if $\mu_c(p)=0$ (resp.\
$\mu_c(p)=\mu'_c(p)=0$).
By \eqref{eq:Vv} and the fact $\hat H=vH$, 
$\mu_c(p)$ is a geometric
invariant called {\it normalized cuspidal curvature}. 
However, $\mu'_c(p)$ does depend on the choice of the
parameter $u$. Since $|f_v(0,0)|=1$, the following formula holds
\begin{equation}\label{eq:kH}
 \mu_c(p)=
  \frac{-\inner{\psi(0,0)}{\nu_u(0,0)}}{|\psi(0,0)\times_g f_v(0,0)|^2}
  =
  \frac{-\inner{f_{uv}(0,0)}{\nu_u(0,0)}}{|f_{uv}(0,0)\times_g f_v(0,0)|^2}.
\end{equation}
The right-hand side of \eqref{eq:kH} is
independent of the choice of an adapted coordinate system.
Moreover, if we take a positively oriented local 
coordinate system
$(u,v)$ satisfying only \ref{item:adapted2:1} 
of Definition \ref{def:adapted2},
then we can write
\begin{equation}\label{eq:kHinv}
 \mu_c(p)=
  \frac{-|f_v(0,0)|^3\inner{f_{uv}(0,0)}
{\nu_u(0,0)}}{|f_{uv}(0,0)\times_g f_v(0,0)|^2},
\end{equation}
which might be useful rather than \eqref{eq:kH}.
The invariant $\mu_c$ plays a similar role as the
cuspidal curvature for non-degenerate singular points
of the first kind. For example,
the following assertion holds (cf.\ Proposition~\ref{lem:kappac-psiccr}).

\begin{prop}\label{prop:second-front}
\begingroup
\renewcommand{\theenumi}{{\rm(\arabic{enumi})}}
\renewcommand{\labelenumi}{{\rm(\arabic{enumi})}}
 Let $f$ be a frontal and $p$ a non-degenerate 
 singular point of the second kind. 
 Then  the following three assertions are equivalent{\rm:}
 \begin{enumerate}
  \item\label{item:second-1} 
       the mean curvature function is rationally bounded at $p$,
  \item\label{item:second-2} 
       $f$ is not a front at $p$,
  \item\label{item:second-3} 
       the normalized cuspidal curvature $\mu_c$
       vanishes at $p$. 
 \end{enumerate}
\endgroup
\end{prop}

\begin{Proof}
 The equivalency of \ref{item:second-1} and \ref{item:second-3}
 has already been mentioned. 
 So it is sufficient to show the 
 equivalency of \ref{item:second-2} and \ref{item:second-3}. 
 Let $(u,v)$ be an adapted coordinate system centered at $p$.
 Since $p$ is a non-degenerate singular point,
 the signed area density function (cf.\ Definition~\ref{def:lambda}) 
 satisfies $(\lambda_u(p),\lambda_v(p))\ne (0,0)$.
 Since $\lambda_u(p)=0$ and $f_u(p)=0$, we have
 \begin{equation}\label{eq:new0}
  0\ne \lambda_v(p)=\det_g(f_{u},f_v,\nu)_v(p)=
   \det_g(f_{uv}(p),f_v(p),\nu(p)).
 \end{equation}
 On the other hand, since $f_u(p)=0$, we have
 \[
  \langle f_v(p), \nu_u(p)\rangle=
  -\langle f_{uv}(p), \nu(p)\rangle=\langle f_u(p), \nu_v(p)\rangle=0.
 \]
 In particular, $f_v(p), \nu(p),\nu_u(p)$ are mutually
 orthogonal in $T_pM^3$.
 Thus, \eqref{eq:new0} implies that
 $\langle f_{uv}(p),\nu_u(p)\rangle \ne 0$ if and only if
 $f_v$ and $\nu_u$ are linearly independent
 (i.e.\ $f$ is a front at $p$), proving the assertion.
\end{Proof}
We set
\begin{equation}\label{eq:muP}
 \mu_\Pi^{}:=\kappa_\nu\mu_c
\end{equation}
and call it the {\it normalized product curvature\/} at $p$.
We now investigate the relationship between $\mu_\Pi^{}$
and the behavior of Gaussian curvature near the
singular set. 
Since the Gaussian curvature $K$ of $f$ satisfies (cf.\ \eqref{eq:cg})
\begin{align*}
\hat K:=v K
=v\frac{h_{11}h_{22}-(h_{12})^2}{g_{11}g_{22}-(g_{12})^2}
+O(v),
\end{align*}
we have the equality
\begin{align}\label{eq:HK}
 \hat K(u,0)&=
\left.\frac{
    -v^2\inner{\psi}{\nu_u}h_{22}-v^2\varepsilon\inner{\psi}{\nu_v}h_{22}
    -v^3\inner{\psi}{\nu_v}^2}{v^2|\psi\times_g f_v|^2}\right|_{v=0}\\
&=\frac{-\inner{\psi}{\nu_\eta}h_{22}}{|\psi\times_g f_v|^2}
=
 2\hat H(u,0) \frac{h_{22}(u,0)}{|f_v(u,0)|^2}
 =2\hat H(u,0)\kappa_{\nu}(u),
 \nonumber
\end{align}
here we used the relation $\kappa_\nu=h_{22}/|f_v|^2|_{v=0}$
obtained by \eqref{eq:kn2} in Proposition~\ref{prop:conti},
where $\kappa_\nu(u)$ is the limiting normal curvature
defined in Section~\ref{sec:prelim}.
Then we have
\begin{align}
\label{eq:K22}
 \hat K(u,0)
            &=2\hat H(u,0)
         \biggl(\kappa_\nu(p)+\hat\kappa'_\nu(p) u +O(u^2)\biggr)\\
\nonumber
&
=
         \biggl(\mu_c(p)+\mu'_c(p) u +O(u^2)\biggr)
         \biggl(\kappa_\nu(p)+\hat\kappa'_\nu(p) u +O(u^2)\biggr)\\
\nonumber
&=
\mu_\Pi^{}(p) +\biggl(\mu_c(p)\hat\kappa'_\nu(p)+\mu'_c(p)
\kappa_\nu(p)\biggr)u+O(u^2).
\end{align}
We remark that 
$\hat\kappa'_\nu(p)$
is the derivative of $\kappa_\nu(u)$
with respect to the non-arclength 
parameter $u$.  (On the other hand,
$\kappa'_\nu(p)$ in \eqref{eq:prod} is the
derivative with respect to the arclength
parameter.)
Since the notion of adapted coordinate system 
is described in terms of first fundamental forms,
the relation \eqref{eq:Vv} implies that
$\hat K(0,0)$ 
is an intrinsic
invariant.  So we get the following 

\begin{prop}
\label{prop:gaussbdd}
 Let $f\colon{}\Sigma^2\to (M^3,g)$ be a frontal
 and $p$ a non-degenerate singular point 
 of the second kind.
 Then the normalized product curvature $\mu_\Pi^{}$
 {\rm(}cf.\ \eqref{eq:muP}{\rm)}
 is an intrinsic invariant. Moreover,
 the Gaussian curvature  $K$ is rationally bounded at $p$
 if and only if $\mu_{\Pi}^{}(p)=0$. 
\end{prop}

Since $\hat\kappa'_\nu(p)=d\kappa_\nu(u,0)/du|_{u=0}$
depends on the parameter $u$,
we consider the co-vector
\[
  \omega_\nu(p):=\hat\kappa'_{\nu}(p)\, du\in T_p^{*}\Sigma^2
\]
instead, which does not depend on the choice of
parameter of the singular curve $\gamma$.
By \eqref{eq:K22},  we also get the following:

\begin{thm}
\label{thm:gaussbdd}
 Let $f\colon{}\Sigma^2\to (M^3,g)$ be a front
 and $p$ a non-degenerate singular point 
 of the second kind.
 Then the Gaussian curvature
 $K$ is rationally bounded {\rm(}resp.\ rationally continuous{\rm)}
 at $p$
 if and only if $\kappa_{\nu}(p)=0$ {\rm(}resp.\ $\kappa_{\nu}(p)=0$
 and $\omega_{\nu}(p)=0${\rm)}.
 Moreover, $K$ is bounded on a neighborhood $\U$
 of $p$ if and only if $\kappa_\nu$ vanishes along
 the singular curve in $\U$.
\end{thm}

\begin{proof}
 The last assertion (i.e., boundedness of $K$ on $\U$)
 follows from \eqref{eq:HK}
 by using the same argument of the 
 the last  assertion of
 Theorem \ref{thm:main1}.
\end{proof}

\begin{rem}\label{rem:tail}
 Suppose that $p$ is a swallowtail singularity 
 satisfying $\kappa_\nu(p)>0$. 
 Let $(u,v)$ be an adapted coordinate system centered at $p$.
 Then the Gaussian curvature 
 takes different signs on $\{v>0\}$ and $\{v<0\}$ near $p$.
 We take the unit normal vector field $\nu$ so that
 the signed area density function $\lambda$ 
 satisfies $\lambda_v(p)>0$.
 A given swallowtail is called positive (resp.\ negative)
 if $\lambda_{\eta\eta}>0$ (resp.\ $\lambda_{\eta\eta}<0$)
 (cf.\ \cite[Section 3]{SUY4}).
 Then the domain $\{v<0\}$ (resp.\ $\{v>0\}$)
 corresponds to the tail part of the swallowtail 
 (that is, whose image has no 
 self-intersections near $p$, see \cite[p.\ 518]{SUY}),
 and so the sign of 
 $-\mu_\Pi^{}$ (resp.\ $\mu_\Pi^{}$)
 coincides with the sign of the Gaussian curvature
 of the tail part near $p$.
\end{rem}

\begin{defi}
 An adapted coordinate system $(u,v)$ at $p$
 is called {\it strongly adapted\/} if
 $f_{uv}$ is perpendicular to $f_v$ at $p$.
\end{defi}
\begin{lem}
 For each singular point $p$ of 
 the second kind, there exists a strongly 
 adapted coordinate system. 
\end{lem}
\begin{Proof}
 For an adapted coordinate system $(u,v)$,
 the new coordinate system $(U,V)$ defined by
 $U:=u$ and $V:=v-\inner{f_{uv}(p)}{f_v(p)}uv$
 gives a strongly adapted coordinate system.
\end{Proof}
Let $(u,v)$ be a strongly adapted coordinate system at
the singular point $p$ of the second kind and take $\psi$
as in \eqref{eq:psi-def}.
Since $\inner{\psi}{\nu}=\inner{f_u+\varepsilon f_v}{\nu}/v=0$
for $v\neq 0$, the continuity yields that $\psi$ is perpendicular
to $\nu$ on the singular set.
Moreover, since $\psi(0,0)=f_{uv}(0,0)$ is linearly independent
to $f_{v}(0,0)$,
$\{f_v,\psi,\nu\}$
gives a frame field near $(0,0)$.
Moreover, $f_v\times_g \psi$ is proportional to the unit normal vector
$\nu$.
\begin{thm}\label{thm:main3}
 Let $f$ be a frontal
 and $p$ its  singular point
 of the second kind,
 and let $\gamma(t)$ be the
 singular curve such that $\gamma(0)=p$.
 If $\gamma(t)$ $(t\ne 0)$ is
 a singular point of the
 first kind, then it holds that
 \[
   |\mu_c(p)|=
   \lim_{t\to 0}
   \frac{|\kappa_c(\gamma(t))|}{2|\kappa_s(\gamma(t))|^{1/2}}.
 \]
 In particular, the product curvature $\kappa_\Pi^{}(\gamma(t))$
 does not converge to  the normalized product curvature
 $\mu_{\Pi}^{}(p)=\kappa_\nu(p)\mu_c(p)$.
\end{thm}

\begin{Proof}
 Let $(u,v)$ be a strongly adapted coordinate system
 and take the null vector field as $\eta=\partial_u+\varepsilon(u) \partial_v$,
 where $\varepsilon(u)\neq 0$ for $u\neq 0$  and $\varepsilon(0)=0$.
 Since $\psi(p)=f_{uv}(p)$ is perpendicular to
 $f_v(p)$, \eqref{eq:kH} reduces to
 \begin{equation}\label{eq:kc_red}
  \mu_c(p)=-\frac{\inner{\psi}{\nu_u(p)}}{|\psi|^2},
 \end{equation}
 and by a choice of $\nu$, it holds at $p$ that
 \begin{equation}\label{eq:fvpsi}
    f_v \times_g \psi 
     = |\psi|\nu.
 \end{equation}%
 By \cite[Page 501]{SUY}, we have that
 \begin{equation}\label{eq:limks}
  \lim_{u\to 0}|\varepsilon(u) \kappa_s(u)|
  =|\det_g(f_v(p),f_{uv}(p),\nu)|\\
  =|\det_g(f_v(p),\psi,\nu)|
  =|\psi|.
 \end{equation}
 On the other hand, we have that
 \[
   f_{\eta\eta}=\psi_{\eta} v+\psi v_{\eta}, \quad
   f_{\eta\eta\eta}=\psi_{\eta\eta} v+2\psi_\eta v_{\eta}+\psi v _{\eta\eta}.
 \]
 Since $v_\eta=v_u+\varepsilon v_v=\varepsilon$,
 we have that
 \[
 f_{\eta\eta}(u,0)=\psi \varepsilon, \quad
 f_{\eta\eta\eta}(u,0)=2\psi_\eta \varepsilon+\psi v _{\eta\eta}.
 \]
 Let $\gamma(u)=(u,0)$ be the singular curve and
 set
 \[
 \hat\eta(u) := \sgn(\varepsilon(u))\eta
              = \sgn(\varepsilon(u))\bigl(\partial_u + \varepsilon(u)\partial_v)
 \]
 for $u\neq 0$.
 Then $\{\gamma',\hat\eta\}$ is positively oriented for each $u\neq 0$,
 and so, the definition \eqref{eq:kc} of the cuspidal curvature
 $\kappa_c$ reduces to
 \begin{align*}
 \kappa_c &= \sgn(\varepsilon)
            \frac{|f_u|^{3/2}\det_g(f_u,f_{\eta\eta},f_{\eta\eta\eta})}
  {|f_u\times_g f_{\eta\eta}|^{5/2}}\\
  &= \sgn(\varepsilon)
  \frac{|\varepsilon f_v|^{3/2}
  \det_g(-\varepsilon f_v,\varepsilon\psi,2\varepsilon\psi_{\eta}+\psi
  v_{\eta\eta})}{|(\varepsilon f_v)\times_g (\varepsilon \psi)|^{5/2}}\\
  & =
  \frac{-2|f_v|^{3/2}\det_g(f_v,\psi,\psi_{\eta})}{%
          \sqrt{|\varepsilon|}|f_v\times_g\psi|^{5/2}  }
    = \frac{-2|f_v|^{3/2}\inner{f_v\times _g\psi}{\psi_{\eta}}}{%
        \sqrt{|\varepsilon|}|f_v\times_g\psi|^{5/2}}.
 \end{align*}
 Thus, by \eqref{eq:limks} and \eqref{eq:fvpsi},
 \begin{align*}
   \lim_{u\to 0}\frac{|\kappa_c|}{2|\kappa_s|^{1/2}}
   &= \lim_{u\to 0} \left|
    \frac{|f_v|^{3/2}\inner{f_v\times_g\psi}{\psi_{\eta}}}{%
    \sqrt{|\varepsilon \kappa_s|} |f_v\times_g\psi|^{5/2}}
    \right|\\
   & = \frac{|\inner{\nu(p)}{\psi_{u}(p)}|}{|\psi(p)|^{2}}
     = \frac{|-\inner{\nu_{u}(p)}{\psi(p)}|}{|\psi(p)|^{2}}
     =|\mu_c(p)|
 \end{align*}
 because $\eta=\partial_u$ at $p$, proving the assertion.
\end{Proof}

We now assume that $p$
is a swallowtail singularity of $f$.
We set (cf.\ \eqref{eq:kc02})
\[
  \tau_s:=
       2\sqrt{2}\lim_{t\to 0}\sqrt{|t|}\,|\kappa_s(\gamma(t))|,
\]
and call it the {\em limiting singular curvature\/}
at $p$, where $t$ is the arclength parameter of
$\hat\gamma(t)=f\bigl(\gamma(t)\bigr)$ for $t\neq 0$.
By definition,
$\tau_s$ is an intrinsic invariant.
We remark that $\kappa_s$ diverges to $-\infty$
at a swallowtail (\cite[Corollary 1.14]{SUY}),
and only the absolute value of $\kappa_s$
is meaningful.
On the other hand,
$\varepsilon'(0)\neq 0$ by Fact \ref{fact:criteria} \ref{item:Fact2},
where $\varepsilon(u)$ is the function as in \eqref{eq:second-eta},
and $'=d/du$.
By 
$f_{uu}(p)=-\varepsilon'(0)f_v(p)$,
it holds that 
\begin{equation}\label{eq:gamma2}
\hat\gamma''(0)\ne0,
\end{equation}
where $\hat\gamma=f\circ\gamma$.
\begin{prop}\label{prop:36}
 Let $f:\Sigma^2\to (M^3,g)$ be a front, and
 $p\in \Sigma^2$ a swallowtail singularity
 and $\gamma(u)$ the singular curve such that $\gamma(0)=p$.
 Then the following identity holds {\rm(}cf.\ \eqref{eq:gamma2}{\rm)}
 \begin{equation}\label{eq:36}
 \tau_s    =
         \frac{|\det_{g}\bigl(\hat \gamma''(0),\hat \gamma'''(0),\nu(p)\bigr)|}{%
           |\hat \gamma''(0)|^{5/2}},
 \end{equation}
 where $'=d/du$ and $\hat\gamma=f\circ\gamma$.
\end{prop}

\begin{Proof}
 Take a strongly adapted coordinate system $(\U;u,v)$ and
 let $t=t(u)$ be the arclength parameter of $\hat\gamma=f(u,0)$
 for $u\neq 0$.
 Since
 the tensor fields
 $\det_g$ and $\inner{~}{~}$ are
 parallel with respect to the Levi-Civita connection
 of $(M^3,g)$, we have
 \begin{align*}
  \lim_{u\to0}
  \frac{\det_g\big(\hat\gamma',\hat\gamma'',\nu(\gamma)\big)}{u^2}
  &=
  \lim_{u\to0}
  \frac{\det_g\big(\hat\gamma',\hat\gamma'',\nu(\gamma)\big)''}{2}\\
  &=
  \lim_{u\to0}
  \frac{\det_g\big(\hat\gamma',\hat\gamma''',\nu(\gamma)\big)'
  +\det_g\big(\hat\gamma',\hat\gamma'',\nu(\gamma)'\big)'}{2}\\
  &=
  \frac{\det_g\big(\hat\gamma'',\hat\gamma''',\nu(\gamma)\big)}{2}\bigg|_{u=0},
 \end{align*}
 and
 \begin{equation}\label{eq:arclength-limit}
  \lim_{u\to0}
   \dfrac{\inner{\hat\gamma'}{\hat\gamma'}}{u^2}
   =
   \displaystyle
   \lim_{u\to0}
   \dfrac{
   \inner{\hat\gamma'}{\hat\gamma'}''}{2}\\[3mm]
   =
   \displaystyle
   \inner{\hat\gamma''}{\hat\gamma''}\big|_{u=0}.
 \end{equation}
 Since
 \[
  |\kappa_s(u)|
   =
   \frac{|\det_g(\hat\gamma'(u),\hat\gamma''(u),\nu(\gamma(u)))|}{%
         |\hat\gamma'(u)|^{3}},
 \]
 we have
 \[
 \lim_{u\to0}|u||\kappa_s(u)|
 =
 \lim_{u\to0}
 \dfrac{|\det_g(\hat\gamma',\hat\gamma'',\nu(\gamma))|}{|u|^2}
 \dfrac{|u|^3}
 {|\hat\gamma'|^3}
 =
 \dfrac{|\det_g(\hat\gamma'',\hat\gamma''',\nu(\gamma))|}
 {2|\hat\gamma''|^3}\bigg|_{u=0}.
 \]
 On the other hand,
 \[
  \lim_{u\to0}
   \frac{|t(u)|}{u^2}
   =
   \lim_{u\to0}
   \left|\frac{t(u)}{u^2}\right|
   =
   \lim_{u\to0}
   \frac{|\hat\gamma'(u)|}{2|u|}
    =
   \lim_{u\to0}
    \frac{1}{2}
   \sqrt{\frac{\inner{\hat\gamma'(u)}{\hat\gamma'(u)}}{u^2}}
    =
    \frac{|\hat\gamma''(0)|}{2}
 \]
holds because of \eqref{eq:arclength-limit}. Thus we have
\begin{equation}\label{eq:tu}
 \lim_{u\to0}
  \frac{\sqrt{|t(u)|}}{|u|}
  =
  \frac{\sqrt{|\hat\gamma''(0)|}}{\sqrt{2}}.
\end{equation}
Hence
\begin{align*}
   2&\sqrt{2}\lim_{t\to0}
   \sqrt{|t|}|\kappa_s(t)|
   =
   2\sqrt{2}\lim_{u\to0}
   \bigg|
   \frac{\det_g\bigl(\hat\gamma',\hat\gamma'',\nu(\gamma)\bigr)}{u^2}\,
   \frac{\sqrt{|t(u)|}}{u}\,
   \frac{u^3}{|\hat\gamma'|^3}
   \bigg|\\
   &=
   2\sqrt{2}\frac{|\det_g(\hat\gamma'',\hat\gamma''',\nu(\gamma))|}{2}\,
   \frac{\sqrt{|\hat\gamma''|}}{\sqrt{2}}\,
   \frac{1}{|\hat\gamma''|^3}\bigg|_{u=0}
   =
   \frac{|\det_g(\hat\gamma'',\hat\gamma''',\nu(\gamma))|}
   {|\hat\gamma''|^{5/2}}\bigg|_{u=0}
 \end{align*}
 proves the assertion.
\end{Proof}

\begin{cor}\label{cor:projection}
 Let $f:\Sigma^2\to \R^3$ be a front, and
 $p\in \Sigma^2$ a swallowtail singularity.
 Let
 $P$ be the
 tangential plane of $f$ at $f(p)$
 {\rm(}that is,
  the plane passing through $f(p)$
 which is orthogonal to $\nu(p)${\rm)},
 and $\sigma$ the orthogonal projection of
 $\hat \gamma:=f\circ \gamma$
 to the plane $P$.
 Then
 $\tau_s$ coincides with the
 absolute value of the cuspidal curvature of the
 curve $\sigma$ in the plane $P$.
\end{cor}

\begin{Proof}
 By taking a strongly adapted coordinate system $(\U;u,v)$,
 it holds that
 $\sigma(u)=\hat\gamma(u)-\inner{\hat\gamma(u)}{\nu(p)}\nu(p)$.
 Since $|\nu(p)|=1$,
 the absolute value of cuspidal curvature of $\sigma$ is
 equal to
 \[
 \left|
 \frac{\det\big(\hat\gamma''(0),\hat\gamma'''(0),\nu(p)\big)}
 {|\sigma''(0)|^{5/2}}\right|.
 \]
 Since
 $\hat\gamma''(u)=f_{uu}(u,0)$
 and
 $f_u(p)=0$,
 we have
 \[
   \inner{f_{uu}(p)}{\nu(p)}=-\inner{f_{u}(p)}{\nu_u(p)}=0.
 \]
 Thus
 $|\sigma''(p)|=|\hat\gamma''(p)|$ holds.
 Hence we have the assertion.
\end{Proof}

We next consider the limit
\[
  \tau_c:
=   \frac{\sqrt{2\sqrt{2}}}{2}
   \lim_{t\to 0}\left||t|^{1/4}\kappa_c(\gamma(t))\right|,
\]
where $t$ is the arclength parameter of $\hat\gamma$.
We call $\tau_c$ the {\it limiting cuspidal curvature}.

\begin{prop}\label{cor:upper}
 If $f$ is a front and $p$ is a swallowtail,
 then  it holds that
 \begin{equation}\label{eq:tauc}
  \tau_c= \sqrt{|\tau_s|}|\mu_c(p)|.
 \end{equation}
 Moreover, the following identity holds{\rm:}
 \begin{equation}\label{eq:tauc2}
  (\sqrt{|\tau_s|}\hat K(p)=)
    \sqrt{|\tau_s|}\kappa_{\nu}(p)\mu_c(p)
    =\sgn(\mu_c(p))\kappa_\nu(p) \tau_c.
 \end{equation}
 In particular, the absolute value of
 the right hand side is intrinsic.
\end{prop}
\begin{Proof}
 Using Theorem \ref{thm:main3} and Proposition \ref{prop:36},
 we get \eqref{eq:tauc}.
 By the definition of $\tau_c$, \eqref{eq:hathexpand}
 and \eqref{eq:HK}, we have
 \[
   \sgn(\mu_c(p))\kappa_\nu(p) \tau_c
    =
   \sqrt{|\tau_s|}\kappa_\nu(p)\mu_c(p)
    =
   2\sqrt{|\tau_s|}\kappa_\nu(p)\hat{H}(p)
   =
   \sqrt{|\tau_s|}\hat{K}(p),
 \]
 which proves the assertion.
\end{Proof}
\begin{exa}\label{exa:sw2}
 Consider a front\footnote{This example
 was suggested by the referee.}
 \[
 f(u,v):= \left(
 v+\frac{u^2}{2}-\frac{b^2u^2v}{2}-\frac{b^2u^4}{8},\ 
 \frac{bu^3}{3}+buv,\ 
 \frac{cv^2}{2}
 \right)
 \qquad (b,c>0)
 \]
 in $\R^3$, where
 $\nu:=
 \bigl(2 b c \left(u^2+v\right),c u \left(b^2 \left(u^2+2
        v\right)-2\right),
    -b \left(b^2 u^2+2\right)\bigr)/\delta$
 and
 \[
    \delta:=
        \sqrt{b^6 u^4+b^4 u^2 
               \left(c^2 \left(u^2+2 v\right)^2+4\right)
               +4 b^2 \left(c^2 v^2+1\right)+4 c^2 u^2}.
 \]
 The $u$-axis consists of the singular set of $f$,
 and $\eta=\partial_u-u\partial_v$ gives a null vector field
 of $f$ on the $u$-axis. Thus
 the origin $(0,0)$ is a swallowtail.
 Other singular points are cuspidal edges.
 We have
 \begin{equation}\label{ex:100}
  2\hat H(u,0)
   =\frac{c}{b^2}+O(u^2),
  \qquad
   \hat K(u,0)
   =-\frac{c^2}{b^2}+O(u^2).
 \end{equation}
 On the other hand,
 we see that
\allowdisplaybreaks{%
 \begin{align}
  \kappa_\nu(u) 
  &=
     \dfrac{\inner{f_{vv}}{\nu}}{|f_v|^2}
  =
  -c+O(u^2),\label{ex:300}\\
  \kappa_s(u)
  &=
   \dfrac{\det(f_u,f_{uu},\nu)}{|f_u|^3}
  =
  -\dfrac{b}{u}+O(u),\label{ex:400}\\
  |\kappa_c(u)|
  &=
  \dfrac{|f_u|^{3/2}
  \left|\det(f_u,f_{\eta\eta},f_{\eta\eta\eta})\right|}
  {|f_u\times f_{\eta\eta}|^{5/2}}
  =
  \dfrac{2c}{b^{3/2}\sqrt{|u|}}+O(u^{3/2}).\label{ex:500}
 \end{align}}%
 Since $(u,v)$ is an adapted coordinate system, 
 using \eqref{eq:kH}, we have
 \begin{equation}\label{eq:mc}
  \mu_c(0)=c/b^2.
 \end{equation}
 In particular,  the constant term 
 of $\hat K$ (cf.\ \eqref{ex:100})
 is equal to the normalized product curvature
 \[
   \mu_\Pi^{}(0)=\kappa_\nu(0)\mu_c(0)=-c^2/b^2.
 \]
 On the other hand,
 by \eqref{ex:400} and \eqref{ex:500},
 we see that 
 \[
 \lim_{u\to0}
 \dfrac{|\kappa_c(u)|}{2|\kappa_s(u)|^{1/2}}
 =
 \lim_{u\to0}
 \dfrac{2c/(b^{3/2}\sqrt{|u|})}{2\left|{b}/{u}\right|^{1/2}}
 =
 \dfrac{c}{b^2},
 \]
 which coincides with $|\mu_c(0)|$
 (cf.\ Theorem \ref{thm:main3}).
 We set
 \[
   t(u):=\int_0^u |f_u(w,0)|dw,
 \]
 then $t$ gives the arclength parameter of
 the image of the singular curve.
 By \eqref{eq:tu} and \eqref{ex:400}, we have  that
 \[
   \tau_s=2\sqrt{2}\lim_{u\to0}\sqrt{|t(u)|}|\kappa_s(u)|
    =
    2\lim_{u\to 0}|u\kappa_s(u)|
    =2b.
 \]
 On the other hand, we have that
 \[
   \left.\frac{\det(f_{uu},f_{uuu},\nu)}{|f_{uu}|^{5/2}}\right|_{(u,v)=(0,0)}=2b,
 \]
 which verifies 
 the formula \eqref{eq:36}.
 We next see that
 \[
  \tau_c =\frac{\sqrt{2\sqrt{2}}}{2}
   \lim_{t\to 0}\left||t|^{1/4}\kappa_c(\gamma(t))\right|
   =
   \frac{\sqrt{2\sqrt{2}}}{2}
   \lim_{u\to 0}\left|\dfrac{|u|^{1/2}}{2^{1/4}}
   \dfrac{2c}{b^{3/2}\sqrt{|u|}}\right|
   =
   \sqrt{2}\dfrac{c}{b^{3/2}}.
 \]
 On the other hand, we have
 $\sqrt{|\tau_s|}|\mu_c(0)|=\sqrt{2}c/{b^{3/2}}$,
 which verifies \eqref{eq:tauc}.
 We also see that $\sqrt{|\tau_s|}\hat K(0,0)=-\sqrt{2}c^2/b^{3/2}$,
 and
 $\kappa_\nu(0)\tau_c=-\sqrt{2}c^2/b^{3/2}$,
 verifying \eqref{eq:tauc2}.
 Finally,  the signed area density function $\lambda$
 (cf.\ Definition \ref{def:lambda}) satisfies
 $
  \lambda_v(0,0)=b~ (>0)
 $.
 By a straightforward calculation,  we have
 \begin{align*}
  \lambda_\eta&=
-\frac{u}{2d} \biggl(b^6 \left(u^4-2 u^2 v\right)
         +b^4 \left(u^2-v\right) 
        \left(c^2 \left(u^2+2 v\right)^2+4\right)  \\
  & \phantom{aaaaaaaaaaaaaaaaaaaaaaaaaaaaaaa} 
  +b^2 \left(8 c^2 v^2+4\right)+4 c^2 \left(u^2-v\right)\biggr)
 \end{align*}
 and
 \[
   \lambda_{\eta\eta}(0,0)=
    \left.\frac{\partial \lambda_\eta}{\partial u}\right|_{(u,v)=(0,0)}
      =-b~(<0).
 \]
 Thus $(0,0)$ is a negative swallowtail and the tail part is $\{v>0\}$
 (cf.\ Remark~\ref{rem:tail}).
 In particular, the sign of the Gaussian curvature of the tail part 
 coincides with that of $\mu_\Pi^{}(0)=-c^2/b^2$, 
 namely, it is negative valued 
 on the tail part near $(0,0)$
 (cf.\ Fig.~\ref{fig:curvedsw}).
\end{exa}

\begin{figure}[hbt!]
\centering
 \includegraphics[width=.4\linewidth]{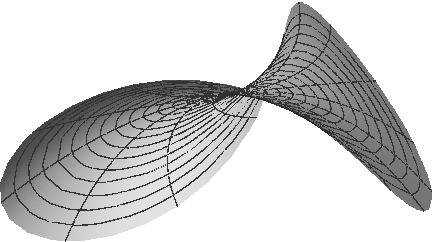}
 \caption{The swallowtail of Example 
 \ref{exa:sw2} with $a=b=1$.}
\label{fig:curvedsw}
\end{figure}

\begin{exa}[Cones of revolution]\label{exa:cone}
Let $\gamma(v)=(x(v),z(v))$ ($|v|<\epsilon$)
be a smooth regular plane curve with 
arclength parameter
such that $x(v)=0$ if and only if $v=0$,
where $\epsilon>0$.
We set
\[
 f(u,v):=(x(v)\cos u,x(v)\sin u,z(v)),
 \qquad (u,v)\in \R/(2\pi\mathbf Z)\times 
 (-\epsilon,\epsilon).
\]
Then
$\nu(u,v):=(-z'(v)\cos u,-z'(v)\sin u,x'(v))$
gives a unit normal vector field of $f$.
The singular set of $f$ is $\{v=0\}$,
and its image is a point, so called
a {\it cone-like singularity}. 
Since the signed area density function
is given by $\lambda(u,v)=-x(v)$,
each singular point is non-degenerate
if and only if $x'(0)\ne 0$.
Since the null vector field of $f$ 
is $\partial_u$,
the singular points $(u,0)$
are all of the second kind.
Moreover, $f$ is a front at $(u,0)$
if and only if $\nu_u\ne 0$,
that is, $z'(0)\ne 0$.
So we now assume $x'(0)z'(0)\ne 0$.
The limiting normal curvature of $f$
at $(u,0)$ is $(\kappa_\nu:=)-x''(0)/z'(0)$,
which coincides with the curvature of 
$\gamma(t)$ at $t=0$.
In particular, the Gaussian curvature $K$
of $f$ is unbounded if $t=0$ is not
an inflection point of $\gamma$.
In fact, $K(u,v)=-x''(v)/x(v)$ 
diverges and changes sign
at $v=0$ when $\kappa_\nu\ne 0$, 
which verifies the third assertion 
of Theorem~\ref{thm:A}.

On the other hand, the normalized 
cuspidal curvature of $f$ is given by 
$(\mu_c:=)z'(0)/x'(0)$,
which does not vanish.
In fact, 
\[
  H(u,v)=\frac12\left(\frac{z'(v)}{x(v)}-
         \frac{z''(v)}{x'(v)}\right)
\]
diverges at $v=0$, 
which verifies 
Proposition \ref{prop:second-front}.
\end{exa}

We now remark that Theorem~\ref{thm:B}
in the introduction  
follows immediately from Corollary 
\ref{cor:cusp} and 
Theorem \ref{thm:gaussbdd}.
Finally, we prove the
Theorem \ref{thm:A} in the introduction.

\begin{tProof}{Proof of Theorem~\ref{thm:A}}
 Let $f\colon{}\U \to (M^3,g)$ be a front and 
 $p\in\U$ a non-degenerate singular point of it.
 Then $p$ is either of the first kind
 or of the second kind.
 Each of these two cases, we can take
 an adapted coordinate system $(u,v)$ centered
 at $p$. Then the signed area density function 
 $\lambda(u,v)$ (cf.\ Definition \ref{def:lambda}) 
 vanishes along the $u$-axis.
 So we can write
 $\lambda(u,v)=v \hat \lambda(u,v)$,
 where $\hat \lambda(u,v)$ is
 a smooth function defined on a sufficiently
 small neighborhood of the $u$-axis,
 and can write
\begin{equation}\label{eq:K3}
   K\, d\hat A
=K \lambda du\wedge dv
=\hat K \hat \lambda\, du\wedge dv,
\end{equation}
 where $\hat K:=vK$
 and
 $K$ is the Gaussian curvature
 of $f$.
 As in the proofs of
 Theorems 
 \ref{thm:main1} and 
 \ref{thm:gaussbdd},
 $\hat K$ is a smooth function
 on a sufficiently small
 neighborhood of
 the $u$-axis, which proves the
 first assertion.
 Moreover, $K\, d\hat A$
 is equal to zero
 only at a point where $\kappa_\nu=0$.
 So we get the equivalency of
 \ref{item:A1} and \ref{item:A3}.

 We next suppose that $\kappa_\nu(p)\ne 0$.
 By the equivalency of
 \ref{item:A1} and \ref{item:A3},
 the equality \eqref{eq:K3}
 yields that
 $\lim_{v\to 0}K(u,v)\lambda(u,v)\ne 0$.
 Since $\lambda(u,v)$ changes sign
 across the $u$-axis,
 we can conclude that
 $K$ is unbounded  
 and changes sign 
 between two sides of the $u$-axis.

 Finally, we consider the case that
 $(M^3,g)$ is the Euclidean 3-space.
 Since $K\,d\hat A$ coincides with the pull-back of the
 area element of the unit sphere by $\nu$,
 as pointed out in Remark \ref{rem:KdA},
 \ref{item:A3} is also equivalent to
 the fact that $p$ is the singular point of $\nu$.
\end{tProof}

We denote by $S^3$ the unit 3-sphere in $\R^4$ centered at the
origin.
For a given front $f:\Sigma^2\to S^3$, 
its unit normal vector field $\nu$
can be considered as a map $\nu:\Sigma^2\to S^3$
using the parallel translations in $\R^4$.
We call this the {\it Gauss map\/} of $f$. 
As a corollary of Theorem \ref{thm:A}, 
we get the following: 

\begin{prop}\label{prop:S3}
 Let $f:\Sigma^2\to S^3$ be a front, and $p\in \Sigma^2$ a
 non-degenerate singular point. Then the 
 Gauss map $\nu:\Sigma^2\to S^3$ of $f$ has a singularity at
 $p$ if and only if the limiting normal curvature $\kappa_\nu(p)$ is
 equal to zero.
\end{prop}

\begin{Proof}
 Since $S^3\subset \R^4$,
 the signed area element of $f$ is written by
 \[
    d\hat A_f=\det(f,f_u,f_v,\nu),
 \]
 where \lq$\det$\rq\ is the determinant function
 on $\R^4$. By using the similar argument
 as in Remark \ref{rem:KdA},
 the Weingarten formula 
 and the Gauss equation 
 $K=1+\det(W)$ yield that
 \[
    d\hat A_\nu=\det(\nu,\nu_u,\nu_v,-f)
        =\det(W) d\hat A_f=(K-1)d\hat A_f,
 \]
 where we took $-f$ as the unit normal 
 vector of $\nu$.
 Since $d\hat A_f$ vanishes at $p$, 
 the equality
 $d\hat A_\nu=(K-1)d\hat A_f=K\, d\hat A_f$
 holds at $p$.
 Then the assertion follows from Theorem \ref{thm:A} and
 the fact that $d\hat A_\nu$ vanishes
 at $p$ if and only if $p$ is a singular point 
 of $\nu$.
\end{Proof}

The hyperbolic space
\[
   H^3:=\biggl\{(t,x,y,z)\in \R^4_1\,;\,
   t^2-x^2-y^2-z^2=1,\,\, t>0\biggr\}
\]
of constant curvature $-1$
is a hyperboloid in the Lorentz-Minkowski $4$-space $\R^4_1$
with signature $(-,+,+,+)$.
Like as in the case of $S^3$,
for an arbitrarily given front $f:\Sigma^2\to H^3$, 
its unit normal
vector field can be considered as the Gauss map 
$\nu:\Sigma^2\to S^3_1$, where
\[
  S^3_1:=\biggl\{(t,x,y,z)\in \R^4_1\,;\,
  t^2-x^2-y^2-z^2=-1\biggr\}
\]
is the de Sitter $3$-space.
If the Gauss map $\nu:\Sigma^2\to S^3_1$
of $f$ is an immersion, then it is
space-like.
We also get the following: 
\begin{prop}\label{prop:H3}
 Let $f:\Sigma^2\to H^3$ be a front, and $p\in \Sigma^2$
 a non-degenerate singular point. Then the Gauss map
 $\nu:\Sigma^2\to S^3_1$ of $f$ has a singularity
 at $p$ if and only if
 the limiting normal curvature $\kappa_\nu(p)$
 vanishes.
\end{prop}

\appendix

\section*{Appendix: The coordinate invariance of blow up}

We give here the procedure of blowing up
and show its
coordinate invariance, which is used 
to define rational boundedness and continuity
in Definition~\ref{def:rational}.

We define the equivalence relation $\sim$
on $\R\times S^1$ by
\[
   (r,\theta)\sim (-r,\theta+\pi),
\]
where $S^1:=\R/2\pi \mathbf Z$.
We set 
$\M:=\R\times S^1/\sim$, namely,
$\M$ is the quotient space of
$\R\times S^1$ by this equivalence relation.
We also denote by
\[
   \pi:\R\times S^1\to \M
\]
the canonical projection.
Let $(\R^2;u,v)$ be the $(u,v)$-plane.
Then there exists a unique
$C^\infty$-map
$\Phi:\M\to \R^2$
such that
\[
   \Phi\circ \pi(r,\theta):=(r\cos\theta,r\sin\theta)
   \qquad ((r,\theta)\in \R\times S^1).
\]
This map $\Phi$ gives the usual
blow up of $\R^2$ at the origin.

From now on, we show that
the coordinate invariance of 
this blow up procedure:
let $(\R^2;U,V)$ be the $(U,V)$-plane,
and consider a diffeomorphism
$f:(\R^2;u,v)\to (\R^2;U,V)$
such that $f(0,0)=(0,0)$.
Then we can write
\[
   f\circ \Phi\circ \pi(r,\theta)
         =(U(r,\theta),V(r,\theta)).
\]
Since $f(0,0)=(0,0)$, it holds that
$U(0,\theta)=V(0,\theta)=0$.
Then the well-known division property
of $C^\infty$-functions yields that
there exist $C^\infty$-function germs
$\xi(r,\theta)$ and $\eta(r,\theta)$
such that
\[
   U(r,\theta)=r \xi(r,\theta),\qquad
   V(r,\theta)=r \eta(r,\theta).
\]
Since $f$ is a diffeomorphism,
one can easily show that
$\xi(0,\theta)^2+\eta(0,\theta)^2$
is positive for all $\theta\in S^1$,
and the $C^\infty$-function
\[
   R(r,\theta):=r \sqrt{\xi(r,\theta)^2+\eta(r,\theta)^2}
\]
is defined on 
$\Omega:=(f\circ\Phi\circ\pi)^{-1}\bigl(\{(U,V)\,;\,
            U^2+V^2<\varepsilon^2\}\bigr)$
for sufficiently small $\varepsilon>0$.
Moreover, there exists a
unique $C^\infty$-function
$\Theta:\Omega\to S^1$
such that
\[
    \cos \Theta(r,\theta)=\frac{\xi(r,\theta)}{%
            \sqrt{\xi(r,\theta)^2+\eta(r,\theta)^2}},
	    \quad
    \sin \Theta(r,\theta)=\frac{\eta(r,\theta)}{%
            \sqrt{\xi(r,\theta)^2+\eta(r,\theta)^2}}.
\]
Then, the $C^\infty$-map
$F:\pi(\Omega)\to \Phi^{-1}
      (\{(U,V)\,;\,U^2+V^2<\epsilon^2\})$
satisfying the property
$ F\circ \pi(r,\theta)=\pi(R(r,\theta),\Theta(r,\theta))$
is uniquely determined, and
satisfies the relation
$F\circ \Phi=\Phi\circ f$.
By our construction, such a map $F$
depends only on $f$.
Hence, by replacing $f$ by $f^{-1}$,
we can conclude that $F$ is a diffeomorphism
for sufficiently small $\epsilon>0$.

Let $p$ be a point on a $2$-manifold $\Sigma^2$. 
The above construction of
$F$ implies that
we can define the \lq blow up\rq\
of the  manifold $\Sigma^2$ at $p$.

\begin{ack}
 The authors thank the referees for careful reading
 and valuable comments.
 The third and the fourth authors thank Toshizumi Fukui for
 fruitful discussions at Saitama University.
 By his suggestion, we obtained the new definition of
 rational boundedness and continuity.
 The second author thanks Shyuichi Izumiya
 for fruitful discussions.
\end{ack}

\end{document}